\documentclass[12pt]{article}

\usepackage[left=2.7cm,bottom=2.7cm,right=2.7cm,top=2.7cm]{geometry}

\usepackage{amsmath,amsfonts}
\usepackage{array}
\usepackage[caption=false,font=normalsize,labelfont=sf,textfont=sf]{subfig}
\usepackage{textcomp}
\usepackage{stfloats}
\usepackage{url}
\usepackage{verbatim}
\usepackage{graphicx}
\usepackage{cite}
\usepackage{mathtools}
\usepackage{algorithm}
\usepackage{algorithmic}
\usepackage{cases}
\usepackage{enumerate}
\usepackage{tabularx}
\usepackage{multirow}
\usepackage{algorithm}
\usepackage{authblk}
\usepackage{algorithmic}
\usepackage{indentfirst}
\usepackage{setspace}
\usepackage{color}
\usepackage[utf8]{inputenc} 
\usepackage[T1]{fontenc}    
\usepackage{hyperref}       
\usepackage{url}            
\usepackage{booktabs}       
\usepackage{amsfonts}       
\usepackage{nicefrac}       
\usepackage{microtype}      
\usepackage{lipsum}		
\usepackage{graphicx}
\usepackage{natbib}
\usepackage{doi}
\usepackage{amssymb}                  
\usepackage{mathrsfs} 
\usepackage[resetlabels]{multibib}
\usepackage[bottom]{footmisc}
\newtheorem{definition}{Definition}
\newtheorem{lemma}{Lemma}
\newtheorem{theorem}{Theorem}
\newtheorem{corollary}{Corollary}
\newtheorem{proof}{Proof}
\newtheorem{remark}{Remark}

\newtheorem{condition}{Condition}
\providecommand{\keywords}[1]{\textit{\quad Key words }:  #1}

\usepackage[utf8]{inputenc}
\usepackage{tikz}
\usetikzlibrary{patterns}
\usetikzlibrary{decorations.pathreplacing,calligraphy}
\usetikzlibrary {patterns.meta}
\usetikzlibrary{calc}
\tikzdeclarepattern{
  name=mylines,
  parameters={
      \pgfkeysvalueof{/pgf/pattern keys/size},
      \pgfkeysvalueof{/pgf/pattern keys/angle},
      \pgfkeysvalueof{/pgf/pattern keys/line width},
  },
  bounding box={
    (0,-0.5*\pgfkeysvalueof{/pgf/pattern keys/line width}) and
    (\pgfkeysvalueof{/pgf/pattern keys/size},
0.5*\pgfkeysvalueof{/pgf/pattern keys/line width})},
  tile size={(\pgfkeysvalueof{/pgf/pattern keys/size},
\pgfkeysvalueof{/pgf/pattern keys/size})},
  tile transformation={rotate=\pgfkeysvalueof{/pgf/pattern keys/angle}},
  defaults={
    size/.initial=5pt,
    angle/.initial=45,
    line width/.initial=.4pt,
  },
  code={
      \draw [line width=\pgfkeysvalueof{/pgf/pattern keys/line width}]
        (0,0) -- (\pgfkeysvalueof{/pgf/pattern keys/size},0);
  },
}
\begin{document}

\title{Sharp minimax optimality of LASSO and SLOPE under double sparsity assumption}

\author[1]{Zhifan Li}
\author[1]{Yanhang Zhang}
\author[1, 2]{Jianxin Yin}

\affil[1]{\footnotesize School of Statistics, Renmin University of China}
\affil[2]{\footnotesize Center for Applied Statistics and School of Statistics, Renmin University of China,
 	\texttt{jyin@ruc.edu.cn} }

\date{}
\maketitle \sloppy

\begin{abstract}
  This paper introduces a rigorous approach to establish the sharp minimax optimalities of both LASSO and SLOPE within the framework of double sparse structures, notably without relying on RIP-type conditions.
  Crucially, our findings illuminate that the achievement of these optimalities is fundamentally anchored in a sparse group normalization condition, complemented by several novel sparse group restricted eigenvalue (RE)-type conditions introduced in this study. We further provide a comprehensive comparative analysis of these eigenvalue conditions.
  Furthermore, we demonstrate that these conditions hold with high probability across a wide range of random matrices. Our exploration extends to encompass the random design, where we prove the random design properties and optimal sample complexity under both weak moment distribution and sub-Gaussian distribution.
  \end{abstract}

\keywords{ double sparsity, minimax optimality, restricted eigenvalue condition, random design.
}

\section{Introduction}

In the past few decades, sparsity has become a fundamental concept in modern statistical learning, particularly when dealing with a large number of covariates relative to the number of observations.
A common scenario involves assuming that only a small proportion of variables significantly influence the response variable, a concept well-studied as individual sparsity \cite{T1996, zhang2010, raskutti2011minimax, bellec2018slope}. However, recent real-world applications have revealed that sparsity patterns can exhibit more complex structures.
For instance, in some cases, the covariates have certain group structures, where groups of variables are either included entirely in the model or excluded altogether, commonly referred to as an "all-in or all-out" selection approach \cite{Y2006, huang2010, tsy2011, zhang2023}.

Beyond group sparsity, another complex sparsity structure called double sparse structure considers sparsity within a group. In this case, it is further assumed that within each selected group, only a small number of its constituent variables are active. This challenging problem is often referred to as sparse group selection and has garnered considerable attention in the fields of machine learning and statistics.
In their respective works, \cite{friedman2010note} and \cite{simon2013sparse} introduced a novel approach known as sparse group Lasso. This method addresses the challenge of sparse group selection by combining the Lasso penalty \cite{T1996} with the group Lasso penalty \cite{Y2006}.
The individual-level Lasso penalty is incorporated to encourage sparsity at the variable level, effectively selecting only a small subset of relevant variables. At the same time, the group Lasso penalty operates at a higher level, promoting sparsity among entire groups of variables, leading to a group selection behavior.

Numerous remarkable approaches have been developed to accelerate the convergence of sparse group Lasso \cite{ida2019fast, zhang2020efficient}. The study conducted by \cite{chatterjee2012sparse} explored the use of sparse group Lasso as a specialized method for investigating regularization with tree hierarchy. Additionally, several studies, including \cite{rao2013sparse}, have successfully applied sparse group Lasso in the context of multitasking learning.
Moreover, \cite{ahsen2017error} presented a comprehensive framework for analyzing error bounds of various techniques, encompassing Group Lasso, sparse group Lasso, and Group Lasso with tree overlap. Furthermore, to improve the consistency of variable selection, \cite{poignard2020asymptotic} introduced the adaptive sparse group Lasso method. These methods have demonstrated promising performance across various applications.
However, it is crucial to emphasize that a noticeable research gap exists concerning the theoretical guarantees of these methods, such as sample complexity and statistical accuracy.

\cite{cai2019sparse} initially established the minimax lower bounds for the estimation error in double sparse linear regression. Subsequently, \cite{li2022minimax} extended these findings to the $\ell_u(\ell_q)$-balls for $u,q \in [0,1]$.
Furthermore, \cite{cai2019sparse} demonstrated the sample complexity and minimax upper bounds for sparse group Lasso, revealing that the method achieves minimax optimality up to a logarithmic term. Notably, the technology employed in \cite{cai2019sparse} is based on the primal-dual witness (PDW) approach \cite{wainwright2009sharp, meinshausen2009lasso}, which plays a critical role in analyzing convex $M$-estimation problems. PDW has been applied in various studies, including logistic regression \cite{wainwright2006high}, Gaussian graphical models \cite{raskutti2008model}, and \cite{ravikumar2010high}.
It is important to note that the use of PDW technology relies on the incoherence condition. While non-correlated conditions are commonly employed in many regression studies, it is worth knowing that, at least within the linear regression framework, the incoherence condition is not necessary \cite{bickel2009simultaneous}.

In the context of high-dimensional linear regression, the restricted isometry property (RIP) \cite{candes2006robust, candes2007dantzig} or the restricted eigenvalue (RE) condition \cite{bickel2009simultaneous} serves as useful tools. \cite{buhlmann2011statistics} provides examples of design matrices that satisfy RIP conditions but do not meet non-dependent conditions. Additionally, \cite{raskutti2011minimax} has demonstrated that in high-dimensional linear regression, the RE condition is required even without considering algorithms solvable in polynomial time. This finding suggests that the RE or RIP condition is necessary for analyzing high-dimensional regression problems, while the incoherence condition may not be required.

The theoretical analysis of sparse group Lasso \cite{cai2019sparse} is challenging due to the RE condition. However, recent work by \cite{bellec2018slope} has provided inspiration by proving the minimax optimality of Lasso using RE conditions. Instead of approaching the problem from the perspective of convex optimization and support set recovery, \cite{bellec2018slope} establishes an upper bound for the convex function based on the oracle inequality of empirical Gaussian complexity, effectively controlling random errors using a "randomness removing tool."
Drawing from this idea and employing a mixed convex penalty formulation, we successfully demonstrated that a convex function composed of mixed penalty forms can also be obtained for the double sparse structures. Specifically, we identified a suitable convex penalty for sparse group Lasso, represented as $\lambda \|\cdot\|_1 + \lambda_g \|\cdot\|_{1, 2}$.

Notably, the work of \cite{bellec2018slope} addresses an open question regarding the minimax optimality of Lasso. While it was previously known that the estimation error of Lasso was minimax sub-optimal (i.e., $\Omega(\frac{s}{n}\log p)$ \cite{bickel2009simultaneous}), it remained uncertain whether sharp optimality could be achieved (i.e., improved to $\Omega(\frac{s}{n}\log \frac{p}{s})$) and under what matrix conditions such optimality would be attainable. \cite{bellec2018slope} answers both questions, presenting significant advancements in the proof techniques for Lasso. 
Taking inspiration by \cite{bellec2018slope}, we demonstrate that sparse group Lasso can achieve the sharp optimality under our proposed RE condition, surpassing the result of \cite{cai2019sparse}. 

Another key focus of \cite{bellec2018slope} is the minimax optimal results for Slope, a regression penalty method proposed by \cite{bogdan2015slope} that incorporates adaptive sparsity requirements. 

We establish the minimax property of Slope under the condition of a correlation matrix, employing a completely different technical approach than the previous proof under an orthogonal design matrix by \cite{su2016slope}. This work builds upon the profound theory of Slope initially proposed by \cite{abramovich2006adapting, abramovich2007optimality} in Gaussian model choice, and it further improves upon the constant order of the minimax result established by \cite{wu2013model}.
\subsection{Our contributions}
The main contributions of our work can be summarized as follows:
\begin{itemize}
  \item We present a novel approach to address the supremum of Gaussian random error under the double sparse parameter space. Without relying on RIP-type conditions, we derive the envelope function $N(u)$ using a sparse group normalization condition. This analysis reveals that the combination of normalization and restricted eigenvalue condition is sufficient for studying sparse group Lasso. 
  \item Leveraging this novel technique, we establish estimation upper bound corresponding to sparse group Slope. By constructing the envelop function $N(u)$, we can derive the non-increasing weighted tuning parameter for double sparse Slope. This method is different from previous analysis about Slope, i.e. \cite{bogdan2015slope, groupslope}, and can be extended to more complicated sparse structure. We also establish the minimax lower bound for double sparsity regression. Follow the previous idea, only condition \ref{sgnorm} is sufficient, which is more general than previous works \cite{cai2019sparse,li2022minimax}.
  \item Besides, we establish the theories under random design matrix. To be precise, we derive the sufficient sample size to provide normalization condition \ref{sgnorm} and eigenvalue condition \ref{ssgre}, \ref{wsgre}. We consider this part under two circumstances: the traditional sub-Gaussian random design and the original weak distribution condition followed by \cite{lecue2017sparse}.
\end{itemize}

\section{Statement of the problem}

\subsection{Model description}
We consider the linear regression problem, which is stated as
\begin{equation*}
	y = X\beta^*+\xi,
\end{equation*}
where $y \in \mathbb{R}^n$ is the response variable, $X \in \mathbb{R}^{n \times p}$ is the design matrix, $\beta^* \in \mathbb{R}^p$ is the underlying coefficient, and $\xi \in \mathbb{R}^n$ is the error term in which each entry draws independently from $\mathcal{N}(0,\sigma^2)$. 

In our settings, the $p$ variables can be divided into $m$ non-overlapping groups.
In particular, $G_j \subseteq [p]$ represents the index set of the $j$-th group and $\cup_{j \in [m]} G_j = [p]$.
Without loss of generality, we assume that each group has the same number of variables, that is, $|G_j| = d$ for each $j \in [m]$.
We say that $\beta$ is $(s,s_0)$-sparse if
\begin{equation*}
	\|\beta\|_{0,2} = \sum_{j=1}^{m}\mathbf{I}(\beta_{G_j}\ne 0) \le s\quad \text{and} \quad\|\beta\|_{0} = \sum_{i=1}^p\mathbf{I}(\beta_i\ne 0)\le ss_0.
\end{equation*}
Here $s$ and $s_0$ are two positive constants that control the sparsity across and within the groups, respectively.
Specifically, $s$ directly imposes a sparsity constraint at group level,
and $s_0$ can be interpreted as the average sparsity per group in the true groups. 
Therefore, $\mathcal{DS}(s,s_0)$ denotes the double sparse parameter space we are interested in:
$$
\mathcal{DS}(s,s_0):=\{\beta\in \mathbb{R}^p:\|\beta\|_{0,2}\le s,~\|\beta\|_{0}\le ss_0 \}.$$

A widely studied approach to sparse group selection is the sparse group Lasso\citep{friedman2010note,simon2013sparse,cai2019sparse}.
The minimization problem of sparse group Lasso is
\begin{equation}\label{sparse group Lasso}
	\hat{\beta} = \arg\min_{\beta\in \mathbb{R}^{p}} \|y-X\beta\|_n^2 + \lambda\|\beta\|_1 + \lambda_g\|\beta\|_{1,2},
\end{equation}
where $\lambda, \lambda_g > 0$ are the tuning parameters.
Here $\|\beta\|_1$ and $\|\beta\|_{1,2}$ correspond to the penalties of the ordinary Lasso \citep{T1996} and group Lasso \citep{Y2006}, respectively.
The specific forms of these norms are defined as
$$
\|\beta\|_{1} = \sum_{i=1}^{p}|\beta_i|\quad \text{and}\quad \|\beta\|_{1,2} = \sum_{j=1}^m \|\beta_{G_j}\|_2.
$$
To achieve the minimax optimality of $\hat \beta$, the relationship between $\lambda$ and $\lambda_g$ in \eqref{sparse group Lasso} should satisfy
$
\lambda_g \asymp \sqrt{s_0}\lambda.
$
Hence the penalty term of sparse group Lasso can be written as
$$
\lambda\left(\|\beta\|_1+\sqrt{s_0}\|\beta\|_{1,2}\right).
$$

\subsection{Notation and preliminaries}\label{notation}

For any positive integer $p$, we denote the set $\{1, 2, \ldots, p\}$ as $[p]$. For any $a, b \in \mathbb{R}$, we use $a \vee b = \max(a, b)$ and $a \wedge b = \min(a, b)$. The floor function $\left \lfloor a \right \rfloor$ is the largest integer no greater than $a$, and the ceiling function $\left \lceil a \right \rceil$  is the
smallest integer no less than $a$.We use the notation $a \precsim b$ to denote the existence of a constant $C$ independent of $n$ such that $a \le Cb$ holds uniformly, and correspondingly, we use $a \succsim b$. If $a \precsim b$ and $a \succsim b$ hold simultaneously, we denote $a \asymp b$.  
Given an $n$-dimensional vector $u$, we define $\|u\|_n^2 \coloneqq \frac{1}{n}\sum_{i=1}^n u_i^2$. 
Given index set $S$, $u_S \in \mathbb{R}^{|S|}$ represents the subvector of $u$ indexed by set $S$, and $X_S \in \mathbb{R}^{n\times |S|}$ represents the submatrix of $X$ indexed by set $S$.
For a matrix $U$, $\|U\|_{op}$ represents the spectral norm of $U$.

In order to analyze element-wise sparsity and group sparsity simultaneously, we use the notation "*" to denote sorting by either the absolute values of elements or the $\ell_2$ norm of vectors. Given a $p$-dimensional vector $u$, which can be arranged as a $d \times m$ matrix $U$, we define the following three types of sorting:

\begin{itemize}
	\item {\bf Element-wise sorting:} Sort all the absolute values of elements in $u$ in descending order. Let $u_{i^*}$ denote the $i$-th largest absolute value in $u$. The sorted vector with absolute values is denoted as $u^*$.
	\item {\bf Group-wise sorting:} Consider the $j$-th group of $u$, i.e., the $j$-th column of $U$. Sort the absolute values of each column in descending order, and the sorted matrix is denoted as $U^*$. For $j$-th column, let $U^*_{i,j}$ denote the $i$-th largest absolute value in the $j$-th column.
	\item {\bf Group sorting:} Given each column $U_j,j\in [m]$ of matrix $U$, sort them in descending order in terms of their $\ell_2$-norm, i.e., $\|U_j\|_2$. Let $\|U_{j^*}\|_2$ denote the $j$-th largest $\ell_2$-norm among these $m$ groups. 
\end{itemize}

To simplify the notations, denote $\phi = \frac{1}{\sqrt{n}}\xi^TX$ and $\Phi \in \mathbb{R}^{d \times m}$ as its matrix form. Denote $\Phi_S = \frac{1}{\sqrt{n}}\xi^T X_S$.
To further analyze the double sparse structure, we consider two types of index sets as follows:
\begin{itemize}
	\item Firstly, randomly select $s$ groups from $m$ groups. Then, choose $s \times s_0$ elements from the selected groups to form a subvector with dimensions $s \times s_0$. The collection of index sets corresponding to these subvectors is referred to as the set family $\mathbb{S}_1(s, s_0)$.
	\item Firstly, randomly select $s$ groups from $m$ groups. Then, choose $s_0$ elements from each selected group, resulting in a subvector with dimensions $s \times s_0$. The index sets corresponding to the subvectors obtained in this manner form a set family denoted by $\mathbb{S}_2(s, s_0)$.	Obviously, $\mathbb{S}_2(s,s_0) \subsetneq \mathbb{S}_1(s,s_0)$.
\end{itemize}

After introducing the notations and problems in Section \ref{notation}, we proceed to control the errors of the error terms in Section \ref{error}.
Building upon the results from Section \ref{error}, we derive the estimation error bounds for sparse group Lasso and sparse group SLOPE in Section \ref{sec:lasso} and Section \ref{sec:slope}, respectively.
Next, we extend our considerations to the case of random design in Section \ref{sec:random}.
Lastly, we discuss and interpret our results in Section \ref{discussion}.

\section{Bound on the stochastic errors}\label{error}
In this section, we consider to bound the stochastic errors based on a sequence of lemmas.
First, we present a sparse group normalization condition for the design matrix $X$, which is crucial for our study of sparse group penalty.
In Section \ref{sec:lower}, we establish the minimax lower bounds for the double sparse linear regression under sparse group normalization condition, 
and show that the upped bounds for sparse group Lasso and sparse group SLOPE are rate-optimal.

\begin{condition}[Sparse Group Normalization]\label{sgnorm}
For each $X_{G_j}$, given any subset $S\subseteq G_j$ with $|S|=s_0$, the submatrix $X_{S}\in \mathbb{R}^{n\times s_0}$ satisfies $\frac{1}{\sqrt{n}}\|X_{S}\|_{op}\le 1$. In other words,
$$
\sup_{j \in [m]}\sup_{|S|=s_0,S\subseteq G_j}\|X_{S}\|_{op}\le \sqrt{n}.
$$
\end{condition}
According to condition \ref{sgnorm}, given $S\subseteq G_j$ for some $j \in [m]$,  we obtain
$$
\|\Phi_{S}\|_{2}^{2}  = \frac{1}{n}\|\xi^TX_{S}\|_2^2\le 	\sup_{j \in [m]}\sup_{|S|=s_0,S\subseteq G_j}\frac{1}{n}\|X_{S}\|_{op}^2\cdot\sigma^2\chi^2_{s_0}.
$$
Therefore, we obtain
\begin{equation*}
	\|\Phi_{S}\|_{2}^{2} \overset{(L)}{\preceq} \sigma^2\chi_{s_0}^2,
\end{equation*}
which implies that for any given $t> 0 $, we have $\mathbf{P}(\frac{1}{\sigma^2}\|\Phi_{S}\|_{2}^{2} \ge t)\le \mathbf{P}(Z\ge t)$, where $Z \sim \chi_{s_0}^2$ is the $\chi^2$ random variable with $s_0$ degrees of freedom.

\begin{remark}
Sparse group normalization is inspired by the group normalization proposed in \cite{tsy2011}. 
When $s_0=d$, sparse group normalization condition is equivalent to the group normalization condition defined in \cite{tsy2011}. 
Therefore, condition \ref{sgnorm} is a less stringent assumption than group normalization.
On the other hand, it boils down to the traditional column normalization used in \cite{bellec2018slope,raskutti2011minimax} when $d = 1$.

Considering the optimal error bounds for group Lasso, \cite{tsy2011} made an important breakthrough by using $\chi^2$ distribution instead of Gaussian to tackle random errors . In the following proofs, the $\chi^2$ distribution of $s_0$ degrees is a key point for obtaining a minimax optimal bound for double sparse structure. Here we only need $\frac{1}{n}\|X_{S}\|_{op}^2$ be an absolute constant, which is assumed to be 1 without loss of generality.
\end{remark}

Assume that $s_0$, $d$, and $m$ are known fixed constants in the following discussion. For a given $s \in [m]$, we consider two forms of subvectors, which can be denoted as
\begin{equation*}
	\Psi_{1}(s) = \sup_{S\in \mathbb{S}_1(s,s_0)}\frac{1}{ss_0\sigma^2}\|\Phi_S\|_2^2	\\
\end{equation*}
and
\begin{equation*}
	\Psi_{2}(s) = \sup_{S\in \mathbb{S}_2(s,s_0)}\frac{1}{ss_0\sigma^2}\|\Phi_S\|_2^2.
\end{equation*}
Next, we derive the upper bound of $\Psi_{1}(s)$ and $\Psi_{2}(s)$ based on condition \ref{sgnorm}.
\begin{lemma}\label{phi}
Assume that condition \ref{sgnorm} holds.
Then, we have
	\begin{equation}\label{s1}
		\mathbf{P}\left\{\Psi_{1}(s) \ge \frac{16}{3}\left(\log\frac{2ed}{s_0}+ \frac{2}{s_0}\log \frac{4em}{s}\right)\right\} \le \frac{s}{4m},
	\end{equation}
	and
	\begin{equation}\label{s2}
		\mathbf{P}\left\{\Psi_{2}(s) \ge \frac{16}{3}\left(\log\frac{2ed}{s_0}+ \frac{2}{s_0}\log \frac{4em}{s}\right)\right\} \le \frac{s}{4m}.
	\end{equation}
\end{lemma}

Similar bounds have been obtained by \cite{zhang2023minimax}. 
However, \cite{zhang2023minimax} derived upper bounds under a more stringent condition, namely, the double sparse RIP condition, which results in a tighter probability tail bound. In our paper, we choose to relax this assumption and solely utilize the sparse group normalization condition.

Based on the random matrix $\Phi$ and $s$, we define two random variables $\upsilon_s$ and $\Upsilon_s$, which are corresponding to $\mathbb{S}_1(s,s_0)$ and $\mathbb{S}_2(s,s_0)$ respectively:

\begin{itemize}
	\item By the rule of $\mathbb{S}_1(s,s_0)$, we obtain $ss_0$ elements each time, and we only focus on the $ss_0$-th largest one. $\upsilon_s$ denotes the largest $ss_0$-th element of $\Phi_S$ for any $ S \in \mathbb{S}_1(s,s_0)$. Consequently, we have $\upsilon_s^2 \le \Psi_{1}(s)$.
	\item By the rule of $\mathbb{S}_2(s,s_0)$, we obtain $s$ groups each time, and we only focus on the $s$-th largest group with $\ell_2$-norm. $\Upsilon_s$ denotes the largest $s$-th group of $\Phi_S$ for any $ S \in \mathbb{S}_2(s,s_0)$, so that we have $\Upsilon_s^2 \le s_0\Psi_{2}(s)$.
\end{itemize}

Now we present the upper bounds for the random variables $\Upsilon_s$ and $\upsilon_s$, which hold uniformly over $s\in[m]$.
Define
\begin{equation*}
	\lambda_{j}  = \sigma\sqrt{\log\frac{2ed}{s_0}+ \frac{2}{s_0}\log \frac{4em}{j}},\quad j \in [m].
\end{equation*}

\begin{lemma}\label{lem42}
	Assume that condition \ref{sgnorm} holds.
	Define event
	$$
	\Omega \triangleq \left\{\max_{j\in [m]}\Upsilon_j\le 4\sqrt{s_0}\lambda_j\right\}\cap \left\{\max_{j\in[m]}\upsilon_j\le 4\lambda_j\right\}.
	$$
	Then, we have
	$$
	\mathbf{P}(\Omega)\ge \frac{1}{2}.
	$$
\end{lemma}
Lemma \ref{lem42} demonstrates that the maximum values of random variables $\Upsilon_s$ and $\upsilon_s$ can be upper-bounded by $\lambda_j$ with large probability.
We extend the sequence $\{\lambda_j\}_{j=1}^m$ defined above to $\{\tilde{\lambda}_i\}_{i=1}^{p}$ as follows: 

\begin{definition}
	Replicate $\lambda_j$ for $s_0$ times, and fill the remaining dimensions with $\lambda_m$ to combine into a $p$-dimensional vector.
	In specific,
	$$
	\tilde{\lambda}_{i} \coloneqq \left\{\begin{array}{ll}
	    \lambda_{\left \lceil \frac{i}{s_0}\right \rceil}, & i\le s_0 \times m,\\
	    \lambda_m, & i > s_0 \times m.
	\end{array}\right.
	$$
\end{definition}
 With this definition, we further define the positive homogeneous function $N(u)$ as follows:
\begin{equation}\label{Nu}
	N(u) =\frac{1}{\sqrt{n}}\left\{\sum_{j=1}^m\|U_{j^*}\|_2\sqrt{s_0}\lambda_j+\sum_{i=1}^pu_{i^*}\tilde{\lambda}_i\right\},
\end{equation}
where $u \in \mathbb{R}^{p}$ and $U \in \mathbb{R}^{d\times m}$ is the matrix form of $u$.

\begin{theorem}\label{gauss2}
		Consider the upper bound of the following Gaussian process given parameter space $N(u)\le 1:$
	\begin{equation}\label{gauss}
		\sup_{u \in \mathbb{R}^p:N(u)\le 1}\left|\frac{1}{n}\xi^TXu\right|.
	\end{equation}
	Under condition \ref{sgnorm}, if event $\Omega$ defined in Lemma \ref{lem42} satisfies $\mathbf{P}(\Omega)\ge \frac{1}{2}$, we have
		$$
	\sup_{u \in \mathbb{R}^p:N(u)\le 1}\left|\frac{1}{n}\xi^TXu\right| \le 4.
	$$
\end{theorem}

\begin{remark}
The key to tackling the complexity the double sparse structure lies in Theorem \ref{gauss2}.
	In the proof of ordinary Lasso or Slope presented in \cite{bellec2018slope}, parameter sorting is straightforward. 
	However, in our case, we must carefully apply the three sorting rules defined in Section \ref{notation} and subsequently concatenate them. 
 
To capture the double sparse structure, we introduce the "tools" $\mathbb{S}_1$ and 
 $\mathbb{S}_2$ in the context of combination. In the proof of Lemma \ref{phi}, we leverage the moment generating function result of $\chi^2(s_0)$, drawing inspiration from \cite{tsy2011} in the framework of group sparsity problems.

Notably, the function $N(u)$ plays a role similar to that in \cite{bellec2018slope}, but to obtain our results requires more sophisticated and intricate techniques, building upon the foundation provided by Lemma \ref{lem42}.
		\end{remark}

We can now present the main concentration inequality for convex optimization under double sparse structure. The following theorem is similar to that of \cite{bellec2018slope} and can be obtained on the basis of Theorem \ref{gauss2} above:

\begin{theorem}\label{gauss3}
	Denote $N(u):\mathbb{R}^{p}\to [0,\infty)$ as a positive homogeneous function, that is, $N(au) = aN(u),\forall u> 0$ and $N(u)> 0, \forall u\ne 0$. For all $\delta_0\in (0,1)$, if the event
	$$
	\Omega_2 \triangleq \left\{\sup_{u \in \mathbb{R}^p:N(u)\le 1}\left|\frac{1}{n}\xi^TXu\right|\le 4\right\}
	$$
	satifies that $P(\Omega_2)\ge \frac{1}{2}$, we have 
	\begin{equation}\label{nu}
		\mathbf{P}\left\{\forall u\in \mathbb{R}^{p}: \frac{1}{n}\xi^TXu\le(4+\sqrt{2})\max\left(N(u),\|Xu\|_n\sigma\sqrt{\frac{\log \frac{1}{\delta_0}}{n}}\right)\right\}\ge 1-\frac{\delta_0}{2}.
	\end{equation}
\end{theorem}

\section{Main results for sparse group Lasso}\label{sec:lasso}

In this section, our goal is to investigate the statistical properties of the estimator $\hat{\beta}$ of sparse group Lasso. 
We establish the upper bounds of estimation error for $\hat \beta$ under our proposed RE-type conditions.

Before formally stating our results, two useful lemmas are presented in the following part. Lemma \ref{lemma:sglasso2} is basic inequality for double sparsity and Lemma \ref{lemma:sglasso} can be applied to universal convex penalized linear regression problems.
Given a $p$-dimensional positive non-increasing sequence $\{\tilde{\lambda}_i\}_{i=1}^p$ such that $\tilde{\lambda}_i\ge \tilde{\lambda}_{i+1}$ for all $i \in [p-1]$. Then, for a $p$-dimensional vector $\beta$, we define
	\begin{equation}\label{Slope1}
		\|\beta\|_{\tilde{\lambda}^*} = \sum_{i=1}^p\tilde{\lambda}_i\beta_{i^*}.
	\end{equation}
	Given a $m$-dimensional positive non-increasing sequence  $\lambda$, we define 
	\begin{equation}\label{Slope2}
		\|\beta\|_{G,\lambda^*}  = \sum_{j=1}^m\lambda_j\|\beta_{G_{j^*}}\|_2,
	\end{equation}
where $\|\beta_{G_{j^*}}\|$ denotes the $j$-th largest $\ell_2$-norm of $\|\beta_{G_j}\|$.

\begin{lemma}\label{lemma:sglasso2}
Let $s\in [m]$ and $s_0 \in [d]$. For any two estimators $\beta, \hat{\beta} \in \mathbb{R}^p$, let $u=\hat\beta-\beta$. If $\|\beta\|_0 \le ss_0$,  we have
\begin{equation}
	\|\beta\|_{\tilde{\lambda}^*}-\|\hat{\beta}\|_{\tilde{\lambda}^*}\le \left(\sum_{i=1}^{ss_0}\tilde{\lambda}_i^2\right)^{\frac{1}{2}}\|u\|_2-\sum_{i=ss_0+1}^p\tilde{\lambda}_iu_{i^*}.
\end{equation}
If $\|\beta\|_{0,2}\le s$, we have
\begin{equation}
	\|\beta\|_{G,\lambda^*}-\|\hat{\beta}\|_{G, \lambda^*}\le \left(\sum_{j=1}^{s}\lambda_j^2\right)^{\frac{1}{2}}\|u\|_2-\sum_{j=s+1}^m\lambda_i\|U_{j^*}\|_2 .
\end{equation}
\end{lemma}

\begin{lemma}\label{lemma:sglasso}
Let $h(\beta): \mathbb{R}^p \to \mathbb{R}$ be a convex function and $\hat{\beta}$ be the solution to the convex optimization problem:
	\begin{equation}\label{sglasso}
		\arg\min_{\beta\in \mathbb{R}^{p}}\left\|y - X\beta\right\|_{n}^2 + h(\beta).
	\end{equation}
Therefore, estimator $\hat{\beta}$ satisfies
	\begin{equation}\label{convex}
		\|X(\hat\beta-\beta^*)\|_n^2 \le \frac{1}{n}\xi^TX(\hat\beta-\beta^*)+ \frac{1}{2}\left(h(\beta^*)-h(\hat{\beta})\right).
	\end{equation}
\end{lemma}

Next, we provide an important ingredient for analyzing the optimality of sparse group Lasso, which we called strong sparse group restricted eigenvalue (SSGRE) condition.
\begin{condition}
[Strong Sparse Group Restricted Eigenvalue Condition]\label{ssgre}
	Given design matrix $X \in \mathbb{R}^{n\times p}$, if $X$ satisfies
	$$
	\theta(s,s_0,c_0) \triangleq \min_{\delta \in \mathcal{C}_{SSGRE}(s,s_0,c_0)\setminus\{0\}}\frac{\|X\delta\|_n}{\|\delta\|_2}>0,
	$$
	where $\mathcal{C}_{SSGRE}(s,s_0,c_0)\triangleq \left\{\delta\in \mathbb{R}^p:\|\delta\|_1+\sqrt{s_0}\|\delta\|_{1,2}\le(2+c_0)\sqrt{ss_0}\|\delta\|_2\right\}$ for some positive constant $c_0$,
	we say $X$ satisfies SSGRE with parameter $\theta(s,s_0,c_0)$.
\end{condition}
Given $\gamma \in (0, 1)$, define 
$$
\delta(\lambda)\triangleq \exp\left(-\left(\frac{\gamma\lambda\sqrt{n}}{(4+\sqrt{2})\sigma}\right)^2\right) \quad \text{so that}\quad \lambda = \frac{(4+\sqrt{2})\sigma}{\gamma}\sqrt{\frac{\log(1/\delta(\lambda))}{n}}.
$$
We set tuning parameter $\lambda^{\#}$ as 
$$
\lambda^{\#} =\frac{(4+\sqrt{2})\sigma} {\gamma\sqrt{n}}\sqrt{\log\frac{2ed}{s_0}+ \frac{2}{s_0}\log \frac{4em}{s}}.
$$
As before discussed, the penalty term of sparse group Lasso is defined as
$$
h(\beta) = 2\lambda^{\#}\left\{\|\beta\|_1+ \sqrt{s_0}\|\beta\|_{1,2}\right\}.
$$

\begin{theorem}\label{sglassomain}
	Assume that $\beta^*$ is $(s, s_0)$-sparse and the conditions in Theorem \ref{gauss3} hold. Assume that condition \ref{ssgre} holds with parameter $\theta(s,s_0,\frac{2(1+\gamma)}{1-\gamma})$.
	The solution of sparse group Lasso $\hat\beta$ satisfies
	\begin{equation*}
		\|\hat\beta-\beta^*\|_2 \le (1+\gamma)\sqrt{ss_0}\lambda^{\#}\max\left\{\frac{2}{\theta(s,s_0,\frac{2(1+\gamma)}{1-\gamma})^2},\frac{1}{2}\left(\frac{\log(1/\delta_0)}{\log(1/\delta(\lambda^{\#}))ss_0}\right)\right\}.
	\end{equation*}
\end{theorem}
Theorem \ref{sglassomain} establishes the estimation upper bounds for sparse group Lasso based on result \eqref{nu} in Theorem \ref{gauss3}.
\begin{remark}
	In Theorem \ref{sglassomain}, we can set
	$$
	\delta_0 = \exp\left\{-C_1\sigma^2\left(ss_0\log\frac{2ed}{s_0}+2s\log\frac{4em}{s}\right)\right\}.
	$$
	Then, $\left(\frac{\log(1/\delta_0)}{\log(1/\delta(\lambda^{\#}))ss_0}\right)$ is of the same order as a constant. 
	Therefore, given a constant $C_1$, there must exist a sufficiently large constant $C_2$ such that we can obtain the sparse group Lasso estimator $\hat{\beta}$ that satisfies, with probability not less than $1-\frac{\delta_0}{2}$,	
	$$
	\|\hat{\beta}-\beta^*\|_2 \le C_2\sigma\sqrt{ss_0\log\frac{2ed}{s_0}+ 2s\log \frac{4em}{s}}.
	$$

\end{remark}

\section{Main results for sparse group Slope}\label{sec:slope}
In this section, we provide the algorithmic construction and corresponding estimation properties of the sparse group Slope. 

Combined with the decreasing sequence $\{\lambda_j\}$ and equations \eqref{Slope1}, \eqref{Slope2}, we define
\begin{equation*}
\|\beta\|_* = \|\beta\|_{\tilde{\lambda}^*}+\sqrt{s_0}\|\beta\|_{G, \lambda}.
\end{equation*}
Similar to the condition proposed in \cite{bellec2018slope}, we propose a Weighted Sparse Group Restricted Eigenvalue Condition (WSGRE):
\begin{condition}[Weighted Sparse Group Restricted Eigenvalue Condition]\label{wsgre}
	Given non-increasing weight sequences $\{\lambda_j\}_{j=1}^m\ \mbox{and}\ \{\tilde{\lambda}_i\}_{i=1}^p$, convex cone $\mathcal{C}_{WSGRE}(s,s_0,c_0)$ is defined as 
	\begin{equation*}
	\mathcal{C}_{WSGRE}(s,s_0,c_0)\triangleq \left\{\delta\in \mathbb{R}^p:\|\delta\|_*\le(2+c_0)\sqrt{\sum_{i=1}^{ss_0}\tilde{\lambda}_i^{2}}\|\delta\|_2,c_0 > 0\right\}.
	\end{equation*}
	Given design matrix $X\in \mathbb{R}^{n\times p}$, if $X$ satisfies
\begin{equation}
\theta(s,s_0,c_0) \triangleq \min_{\delta \in \mathcal{C}_{WSGRE}(s,s_0,c_0)\setminus{0}}\frac{\|X\delta\|_n}{\|\delta\|_2}>0,
\end{equation}
we say that $X$ satisfies WSGRE condition with parameter $\theta(s,s_0,c_0)$.
\end{condition}
Based on the relationship between $\{\lambda_j\}$ and $\{\tilde{\lambda}_i\}$, we have $\tilde{\lambda}_{ss_0} = \lambda_s$ and $\sum_{i=1}^{ss_0}\tilde{\lambda}_i = s_0\sum_{j=1}^{s}\lambda_j$. 
Assume that $\delta \in \mathcal{C}_{SSGRE}(s,s_0,c_0)$, i.e., $\|\delta\|_1+\sqrt{s_0}\|\delta\|_{1,2} \le (2+c_0)\sqrt{ss_0}\|\delta\|_2$.
Observe that
\begin{align*}
	\sqrt{s_0}\sum_{j=s+1}^m\lambda_j\|\delta_{G_{j^*}}\|_2+\sum_{i=ss_0+1}^p\tilde{\lambda}_i\delta_{i^*}&\le \lambda_s\left\{\sum_{j=s+1}^m\sqrt{s_0}\|\delta_{G_{j^*}}\|_2+\sum_{i=ss_0+1}^p\delta_{i^*}\right\}\\
	&\le \lambda_s\left\{\|\delta\|_1+\sqrt{s_0}\|\delta\|_{1,2}\right\}\\
	&\le (2+c_0)\sqrt{ss_0}\lambda_{s}\|\delta\|_2.
\end{align*}
According to the Cauchy-Schwarz inequality, we have
\begin{align*}
	\|\delta\|_{*} \le& \sqrt{s_0}\sum_{j=1}^s\lambda_j\|\delta_{G_{j^*}}\|_2+\sum_{i=1}^ss_0\tilde{\lambda}_i\delta_{i^*}+\\
	&\sqrt{s_0}\sum_{j=s+1}^m\lambda_j\|\delta_{G_{j^*}}\|_2+\sum_{i=ss_0+1}^p\tilde{\lambda}_i\delta_{i^*}\\
	\le& 2\sqrt{\sum_{i=1}^{ss_0}\tilde\lambda_i^2}\|\delta\|_2 + (2+c_0)\sqrt{ss_0}\lambda_{s}\|\delta\|_2\\
	\le& (4+c_0)\sqrt{ss_0}\lambda_{s}\|\delta\|_2.
\end{align*}
Above discussion demonstrates that for any $\delta \in \mathcal{C}_{SSGRE}(s,s_0,c_0)$, it satifies that $\delta \in \mathcal{C}_{WSGRE}(s,s_0,2+c_0)$. 
In other words, the WSGRE condition is more stringent than the SSGRE condition. 
Consider the penalty of sparse group Slope as 
\begin{equation*}
	h(\beta)= \frac{2(4+\sqrt{2})}{\sqrt{n}\gamma}\|\beta\|_*,
	\end{equation*}
where $\gamma$ is some constant belonging to $(0, 1)$.
We establish the estimation upper bound for sparse group Slope in the following theorem.
\begin{theorem}\label{sgSlopemain}
	Assume that $\beta^*$ is $(s, s_0)$-sparse and the conditions in Theorem \ref{gauss3} hold. Assume that condition \ref{wsgre} with parameter $\theta(s,s_0,\frac{2(1+\gamma)}{1-\gamma})$.
	The solution of sparse group Slope $\hat\beta$ satisfies
\begin{equation}
	\|\hat\beta-\beta^*\|_2 \le (1+\gamma)\sqrt{ss_0}\lambda^{\#}\max\left\{\frac{2}{\theta(s,s_0,\frac{2(1+\gamma)}{1-\gamma})^2},\frac{1}{2}\left(\frac{\log(1/\delta_0)}{\log(1/\delta(\lambda^{\#}))ss_0}\right)\right\}.
\end{equation}
where $\lambda^{\#}$ and $\delta(\lambda)$ are defined in the same manner as in Theorem \ref{sglassomain}.
\end{theorem}
The conclusion of Theorem \ref{sgSlopemain} is similar to the conclusion of Theorem \ref{sglassomain} for sparse group Lasso. However, the matrix condition in Theorem \ref{sgSlopemain} is stronger compared to the condition used in Theorem \ref{sglassomain}. The advantage of sparse group Slope over sparse group Lasso is akin to the advantage of ordinary Slope over Lasso, where it becomes adaptive to the unknown sparsity $s$.
The reason behind this "adaptivity to $s$" is that the RE condition has been slightly strengthened for sparse group Slope.
\section{Minimax lower bound}\label{sec:lower}
In this section, we provide the minimax lower bounds for $\ell_2$-parameter estimation on the parameter space $\Theta^{m,d}(s, s_0)$ and state that the upper bounds established in Theorem \ref{sglassomain} and \ref{sgSlopemain} match the lower bounds.
Previous works have extensively studied minimax rates for high-dimensional sparse linear regression. Several papers focus on the element-wise $s$-sparsity class, such as \cite{raskutti2011minimax,verzelen2012minimax, bellec2018slope}, while there have been efforts dedicated to group sparsity, as seen in \cite{tsy2011}.

Though a similar minimax lower bounds of double sparse regression has been provided in \cite{cai2019sparse,zhang2023minimax},
we derive the same lower bound solely by relying on condition \ref{sgnorm}, without any RE or RIP conditions.
We begin with the definition of the following parameter subspace.
Consider parameter space $\widetilde\Theta^{m,d}(s, s_0)$:
$$
\widetilde\Theta^{m,d}(s, s_0) = \{\theta \in \mathbb{R}^p :\|\beta\|_{0,2}\le s\ \text{and}\ \|\beta_{G_j}\|_0\le s_0,\forall j\in [m]\}.
$$ 
Different from $\Theta^{m,d}(s, s_0)$, $ \widetilde\Theta^{m,d}(s, s_0)$ requires each group to be constrained to an $\ell_0$-ball with radius $s_0$, and the group sparsity of $ \widetilde\Theta^{m,d}(s, s_0)$ reaches $s$.
It is obvious that $ \widetilde\Theta^{m,d}(s, s_0)  \subseteq\Theta^{m,d}(s, s_0)$.
Hence we have
\begin{equation}\label{lower_relation}
	\inf_{\hat\beta}\sup_{\beta^* \in \Theta^{m,d}(s, s_0)}\mathbf{E}\|\hat\beta-\beta^*\|^2 \geq \inf_{\hat\beta}\sup_{\beta^* \in \widetilde\Theta^{m,d}(s, s_0)}\mathbf{E}\|\hat\beta-\beta^*\|^2.
\end{equation}
In what follows, we consider the minimax lower bounds for $\Theta^{m,d}(s, s_0)$. 
We first provide the lower bounds for the packing number of $ \widetilde\Theta^{m,d}(s, s_0)$.
Let $M(\rho;\widetilde\Theta^{m,d}(s, s_0), \|\cdot\|_H)$ be the cardinality of $\rho$-packing set of parameter space $\widetilde\Theta^{m,d}(s, s_0)$ in the sense of Hamming metric $\|\cdot\|_H$.
\begin{lemma}[Lower bounds for the packing number \cite{li2022minimax}]\label{lem2}
	The cardinality of $ \frac{ss_0}{4}$-packing set of $\widetilde\Theta^{m,d}(s, s_0)$ is lower bounded as
	\begin{equation*}
		M(\frac{ss_0}{4};\widetilde\Theta^{m,d}(s, s_0), \|\cdot\|_H) \geq \exp\left\{\frac{1}{4}\left(s \log \frac{em}{s}+ss_0\log \frac{ed}{s_0}\right)\right\}.
	\end{equation*}
\end{lemma}
\cite{li2022minimax} utilized the structures of double sparsity and combined multi-ary Gilbert-Varshamov bounds \cite{gilbert1952comparison} to construct the packing set of $\widetilde\Theta^{m,d}(s, s_0)$ in a more concise manner.

The way to prove minimax lower bound for double sparse structure in \cite{li2022minimax,cai2019sparse} needs the eigenvalue condition for design matrix. 
To prove the same lower bound without this assumption, we construct a signed version of packing set by the following lemma on the basis of Lemma \ref{lem2} and sparse group normalization. 
Denote
\begin{equation}\label{sgnorm2}
	\sup_{j\in [m]}\sup_{|S|=s_0,S\subseteq G_j}\|X_{S}\|_{op} = \sqrt{n} \vartheta_{\max}.
\end{equation}

\begin{lemma}\label{lem3}
	Assume that the design matrix $X\in \mathbb{R}^{n\times p}$  satisfies \eqref{sgnorm2}. Then, there exists a subset $\mathcal{M} \subseteq \widetilde\Theta^{m,d}(s, s_0)\cap \{-1,0,1\}^{p}$ satisfying the following properties:
	\begin{itemize}
		\item [(i)] $\|\beta\|_{0,2}= s\ \text{and}\ \|\beta_{G_j}\|_0= s_0,\forall j\in [m]$.
		\item [(ii)]$\|X\beta\|_n^2 \le \vartheta_{\max}^2 ss_0, \forall \beta \in \mathcal{M}$.
		\item [(iii)]$\|\beta_i,\beta_j\|_H \ge \frac{ss_0}{4}, \forall i\ne j\ \text{and}\ \beta_i,\beta_j \in \mathcal{M}$.
	\end{itemize}
\end{lemma}

Combining Lemma \ref{lem3} , we establish the minimax lower bound in the following theorem.

\begin{theorem}\label{thm:lower}
	Consider linear regression model $y = X\beta^* + \varepsilon$, where $\varepsilon \sim \mathcal{N}(0,\sigma^2\mathrm{I}_n)$. 
	Assume that $X$ satisfies \eqref{sgnorm2} with $\vartheta_{\max} < \infty$.
	Then, we have
	\begin{equation}\label{lower bound}
		\inf_{\hat\beta}\sup_{\beta^* \in \Theta^{m,d}(s, s_0)}\mathbf{E}\|\hat\beta-\beta^*\|^2\geq \frac{\sigma^2}{256n\vartheta_{\max}^2}\left( s \log \frac{em}{s}+ss_0\log \frac{ed}{s_0}\right).
	\end{equation}
\end{theorem}
Theorem \ref{thm:lower} establishes the lower bounds for the $\ell_2$ estimation errors, which are consistent with the results in \cite{cai2019sparse,zhang2023minimax}.
Notably, drawing inspiration from the work of \cite{bellec2018slope}, we establish the lower bound by exclusively leveraging the sparse group normalization condition \eqref{sgnorm2}. 

\section{Random design}\label{sec:random}
This section implies some random design conclusions for independent samples. This part is divided into two parallel subsections: The first part is focused on weak random design that is derived by \cite{lecue2017sparse}, which gets rid of the need for specific distribution, such as sub-Gaussian random design assumption. This property is only corresponded to sparse group Lasso. The second part is for traditional sub-Gaussian random design. It is different from \cite{zhou2009restricted} that we do not need RIP-type condition but only normalization condition \ref{sgnorm} combined with restricted eigenvalue \ref{ssgre} or \ref{wsgre}. We establish random design property for sparse group Slope mainly, and the conclusion for sparse group Lasso holds similarly. In one words, we clarify that the conditions \ref{sgnorm},  \ref{ssgre} and  \ref{wsgre} that we used in the context of fix design can hold with a high probability when considering some random design circumstances.

\subsection{Weak random design for sparse  group Lasso}
The  restricted eigenvalue condition \ref{ssgre} for sparse group Lasso  in on the basis of strong RE defined by \cite{bellec2018slope}, where they define the convex cone for sparsity parameter space $\mathcal{S}_{0}(ss_0) = \{\beta\in \mathbb{R}^{p}:\|\beta\|_{0}\le ss_0\}$:   
$$
\mathcal{C}_{SRE}(ss_0, c_1) \triangleq \{\delta\in \mathbb{R}^p:\|\delta\|_1\le(1+c_1)\sqrt{ss_0}\|\delta\|_2,c_1>1\}.
$$

Lasso can obtain the optimal estimation rate by the use of RE on $\mathcal{C}_{SRE}(ss_0,c_1)$, i.e., $ss_0$-sparsity. On the other hand, we can define another convex cone for group sparsity $\mathcal{S}_{0,2}(s) = \{\beta\in \mathbb{R}^p:\|\beta\|_{0,2}\le s\}$:
$$
\mathcal{C}_{SGRE}(s, c_2) \triangleq \{\delta\in \mathbb{R}^p:\|\delta\|_{1,2}\le(1+c_2)\sqrt{s}\|\delta\|_2,c_2>1\}.
$$

Considering the parameter space $\mathcal{S}_0(ss_0) \cap \mathcal{S}_{0,2}(s)$, the convex cone $\mathcal{C}_{SSGRE}(s,s_0,c_0)$ plays a important role in condition \ref{ssgre}. Let $c_0 =c_1+c_2$, we have
$$
\mathcal{C}_{SRE}(ss_0, c_1) \cap  \mathcal{C}_{SGRE}(s,c_2) \subseteq \mathcal{C}_{SSGRE}(s,s_0,c_0).
$$
In addition, we have
$$
\mathcal{C}_{SSGRE}(s, s_0, c_0)\subseteq \mathcal{C}_{SRE}(ss_0,1+c_0)\cap \mathcal{C}_{SGRE}(s,1+c_0).
$$
Therefore, neglecting the constant factor, we can approximate the cone $\mathcal{C}_{SSGRE}$ using the form $\mathcal{C}_{SRE} \cap \mathcal{C}_{SGRE}$. This approximation is analogous to the fact that $\mathcal{DS}(s,s_0) = \mathcal{S}0(ss_0)\cap \mathcal{S}{0,2}(s)$.

On the other hand, \cite{bellec2018slope,lecue2017sparse} demonstrate a delicate conclusion that RE condition and the sparse eigenvalue are equivalent, and we can trivially obtain the equivalence between the group RE condition and group sparse eigenvalue condition. We can make a conjecture there exists a double sparse version equivalence between condition \ref{ssgre} and double sparse eigenvalue condition.
Without loss of generality, in this section, we assume design matrix $X$ is normalized by the factor $\frac{1}{\sqrt{n}}$, and we give a formal version of double sparse eigenvalue condition: 
\begin{condition}\label{dsre}
 If the design matrix $X$ satisfies
	$$
	\theta \triangleq \min_{\beta \in \mathcal{DS}(s,s_0)}\frac{\|X\beta\|_2}{\|\beta\|_2}>0,
	$$
then $X$ is said to satisfy the double sparse RE condition $DSRE(s,s_0,\theta)$.
\end{condition}

We make a conclusion by the following theorem to demonstrate the equivalence between double sparse eigenvalue condition and strong sparse group RE condition. The technique here is based on Maurey's empirical method, which has been used in \cite{lecue2017sparse,oliveira2016lower}:
\begin{theorem}\label{thm:ssgre}
Assume that $X$ satisfies condition \ref{sgnorm} and condition \ref{dsre}. Then, there exist a $\tilde{s}$ of the same order as $s$, and a constant $\tilde{\theta}$ of the same order as $\theta$ that make $X$ satisfy  the condition $SSGRE(\tilde{s},s_0,c_0,\tilde{\theta})$.
\end{theorem}
Theorem \ref{thm:ssgre} implies that the combination of conditions \ref{sgnorm} and \ref{dsre} can deduce condition \ref{ssgre}. Specifically, under i.i.d. random design, if we verify the corresponding conditions that satisfy \ref{sgnorm} and \ref{dsre}, then condition \ref{ssgre} naturally holds as a consequence of Theorem \ref{thm:ssgre}.

A useful condition for the sparse eigenvalue condition is the small ball condition \cite{koltchinskii2015bounding,lecue2017sparse,mendelson2015learning}, and \cite{bellec2018slope} follows it. 
There naturally exists a double sparse version of the small condition and double sparse RE condition.

\begin{theorem}[Small ball condition]\label{thm:smball}
A random vector $x$ valued in $ \mathbb{R}^p$ is said to satisfy the small ball condition over $\mathcal{DS}(s,s_0)$ if there exist positive constants $\tau$ and $\theta_{min}$ such that
\begin{equation}\label{smball}
	\mathbf{P}\left(|\langle \beta, x\rangle|\ge \theta_{\min}\|\beta\|_2\right) \ge \tau.\quad \forall \beta \in \mathcal{DS}(s,s_0).
\end{equation}
Let $X$ be the random design matrix with i.i.d. rows that have the same distribution as $x$ satisfying \eqref{smball} and the sample size satisfies $ n \ge \{\frac{C}{\tau^2}(s\log \frac{em}{s}+ss_0\log\frac{ed}{s_0})\}$, then there exists a constant $\theta \ge \theta_{min}/\sqrt{2}$ that makes $DSRE(s,s_0,\theta)$ hold for $X$ with probability greater than $1-\exp(-C'n\tau^2)$, where $C$ and $C'$ are absolute constants.
\end{theorem}
 \begin{remark}
 	Theorem \ref{thm:smball} is no more than a generalization of Corollary 2.5 in \cite{lecue2017sparse}. We just need to replace the VC-dimension of ordinary sparsity  with double sparsity.
 \end{remark}

To fulfill the condition \ref{sgnorm} for random design, we also follow the insight in \cite{lecue2017sparse} and derive a double sparse version. We define the family of index set
\begin{equation}\label{normset}
	\mathcal{S} = \left\{S:|S| = s_0\ \text{and}\  \exists i, S\subseteq G_i\right\}
\end{equation}
On the basis of  \eqref{normset}, we define $s_0$-dimensional random vector $X_j , j \in \mathcal{S}$. We also define the $\frac{1}{\sqrt{s_0}}$-scaled duel form random variable $X^{*}_j$ induced by $X_j,~ j \in \mathcal{S}$ respectively, i.e.,
\begin{equation}\label{dual}
	X^*_{j} \coloneq \sup_{\|t\|_2 \le \frac{1}{\sqrt{s_0}}}|\langle X_j, t\rangle| = \frac{1}{\sqrt{s_0}}\|X_j\|_2,~\forall j\in \mathcal{S}.
\end{equation}

In other words, $X^*_{j}$ is the operator norm of $\frac{1}{\sqrt{s_0}}X_j$ induced by $\|\cdot\|_2$. Condition \ref{sgnorm} is exactly the upper bound of  $\max_{j \in \mathcal{S}}\frac{1}{n}\sum_{i=1}^n(X^*_{ij})^2$ with a high probability, where $X^*_{ij}$ is the $i$-th copy of $X^*_{j}$.

\begin{theorem}[Weak moment condition]\label{weak}
	Given random variable $Z$, $\|Z\|_{L_q}$ denotes $\{\mathbf{E}(|Z|^q)\}^{\frac{1}{q}}$. Let $q_0 = C_1(\log m + s_0 \log \frac{ed}{s_0})$ , where $C_1>1$ is an absolute constant. 
	Assume that for any $j \in \mathcal{S}$,  $\|X_{j}^{*}\|_{L_2} \le 1$, and   $ \forall 2 \le q \le  q_0$,   the $\|X_{j}^{*}\|_{L_q}$ satisfies the growth rate 
	$\|X_{j}^{*}\|_{L_q} \le \kappa_1 q^{\alpha}$ where $\alpha\ge \frac{1}{2}$. 
	Assume that the sample size
	$
	n \ge C_2(\alpha)\kappa_1^2 q_0^{4\alpha - 1},
	$
	where $C_2(\alpha) = 4\alpha\exp(4\alpha -1)$ is a constant only depend on $\alpha$ and larger than 1. Then, the sparse group normalization condition \ref{sgnorm} holds with probability greater than $1-\exp\{(C_2-1)(\log m + s_0\log\frac{ed}{s_0})\}$.
\end{theorem}

If we consider sub-Gaussian random design, for any $q \ge 2$, we have the growth rate $\|X_{j}^{*}\|_{L_q} \le \kappa_1 q^{\frac{1}{2}}$. So that the condition in Theorem \ref{weak} is weaker because we only need the growth rate hold for $q \le q_0$. $q_0$ exists because we use a truncate technique in the proof. Tight sufficient sample size $n$ is another key point in random design scenery, and if we set $\alpha = \frac{1}{2}$, we derive $n\asymp q_0$, which exactly meets the sample size lower bound we need for sub-Gaussian random design. Therefore, Theorem \ref{weak} delivers an exact relationship between moment growth rate and sample size, which help us tackle different random design setting, such as sub-Gaussian for $\alpha=\frac{1}{2}$ and sub-exponential for $\alpha = 1$.

Therefore, combining Theorem \ref{weak} and Theorem \ref{thm:smball}, we achieve condition \ref{sgnorm} and \ref{ssgre} with probability higher than $1-\exp\{-C_1(s\log \frac{em}{s}+ss_0\log\frac{ed}{s_0})\}-\exp\{C_2(\log m + s_0\log\frac{ed}{s_0})\}$. Moreover, if we set $\alpha = \frac{1}{2}$ in Theorem \ref{weak}, where sub-Gaussian design is a special case, we derive a tight sample size rate $ n \geq C_3(s\log \frac{em}{s}+ss_0\log\frac{ed}{s_0})+C_4(\log m + s_0 \log \frac{ed}{s_0})$.

\subsection{Sub-Gaussian random design}
In this section, we will discuss the sample complexity needed to satisfy the sparse group normalization condition \ref{sgnorm} and the restricted eigenvalue conditions \ref{ssgre} and \ref{wsgre} under the assumption of sub-Gaussian random design. Moreover, we derive the results respectively on the basis of the covariance matrix condition. For this, we need to introduce some definitions, which have been introduced \cite{zhou2009restricted,mendelson2008uniform}.
\begin{definition}
The $p$-dimensional random variable $Y$ is called an isotropic distribution if it satisfies that for any $y \in  \mathbb{R}^p$, there is $\mathbf{E}|\langle Y,y\rangle|^2 =\|y\|_2^2$. 
Furthermore, the $\psi_2$ norm of $Y$ with a constant $\alpha$ is defined as
$$
\|Y\|_{\psi_2} \triangleq \inf\{t:\mathbf{E}\left\{\exp(\langle Y,y\rangle^2/t^2)\right\}\le 2\}\le \alpha\|y\|^2.
$$
\end{definition}

An important example of an isotropic, sub-Gaussian random vector is the Gaussian random vector $Y = (h_1,h_2,\cdots h_p)$ where $h_i,\forall i\in [p]$ are independent $\mathcal {N}(0, 1)$ random variable. Another example is the Bernoulli random vector $Y = (\epsilon_1,\cdots , \epsilon_p)$, where $\epsilon_i,\forall i\in [p]$ are independent, symmetric $\pm 1$ Bernoulli random variables .

Suppose $\zeta_1,\zeta_2,\cdots,\zeta_n$ are independent and identically distributed $p$-dimensional isotropic, sub-Gaussian random vectors, forming a random matrix $Z\in  \mathbb{R}^{n\times p}$ , whose row $i$ is denoted by $\zeta_i$. In this paper, we consider a random design matrix $X$, which is generated as follows:
\begin{equation}\label{randomdesign}
X := Z\Sigma^{\frac{1}{2}},
\end{equation}
where $\Sigma$ is the covariance matrix. That is, we need to set an appropriate  $\Sigma$ so that the empirical design matrix $X$ can satisfy the RE condition and normalization condition we need. Specifically, according to the theoretical framework of \cite{zhou2009restricted,mendelson2008uniform}, given the vector space $\mathcal{V}\in  \mathbb{R}^p$, the key point is  to construct the restricted isometric properties between $Xv$ and $\Sigma^{\frac{1}{2}}v$, and then the condition for empirical design matrix $X$ can be transformed into the corresponding condition for $\Sigma$. In order to achieve this goal, empirical process technique makes an important role, and we give the definition of Gaussian complexity at first:
\begin{definition}
Given a subset $\mathcal{V}\in  \mathbb{R}^p$, we define the Gaussian complexity of $V$ as follows:
$$
\ell^*(\mathcal{V}) = \mathbf{E}\sup_{\delta\in \mathcal{V}}\left|\sum_{i=1}^pg_i\delta_i\right|,
$$
where $\delta_i$ is each component of vector $\delta$, and $g_1, g_2, \cdots, g_p$ are independent $\mathcal{N}(0,1)$ distributions. In particular, given a nonnegative definite matrix $\Sigma$, we define:
$$
\tilde{\ell}^*(V) = \ell^*(\Sigma^{\frac{1}{2}}V) = \mathbf{E}\sup_{v\in V}\left|\langle \Sigma^{\frac{1}{2}}v, g\rangle\right| = \mathbf{E}\sup_{v\in V}\left|\langle v, \Sigma^{\frac{ 1}{2}}g\rangle\right|.
$$
\end{definition}
According to the homogeneity of the norm, we only need to consider the subset of the unit ball sphere $S^{p-1}$, which is defined as:
$$
S^{p-1} \triangleq \left\{v\in  \mathbb{R}^p: \|v\|_2 = 1\right\}.
$$
The main technique we use is the following empirical process result:
\begin{lemma}[Theorem 2.1 in \cite{mendelson2008uniform}]\label{mendelson1}
Let $1 \leq n \leq p$ and $0 < \theta < 1$. Let $\zeta$ be an isotropic sub-Gaussian random vector on $ \mathbb{R}^p$, and $\psi_2$ constant be $\alpha$. And $\zeta_1,\dots,\zeta_n$ are independent copies of $\zeta$. Let $X$ be the random matrix defined in \eqref{randomdesign}, and let $\{\Sigma^{\frac{1}{2}}v,v\in\mathcal{V}\} \subset S^{p -1}$. If sample size $n$ satisfies
$$
  n > c'\alpha^4 \theta^2 \tilde{\ell}^*(\mathcal{V})^2
  $$
  Then with probability of at least $1 - \exp(-\bar{c}\theta^2n/\alpha^4)$, for all $v \in \mathcal{V}$, we have
$$
1 - \theta \leq \| X v \|_2 / \sqrt{n} \leq 1 + \theta,
$$
where $c',\bar{c} >0$ is an absolute constant.
\end{lemma}
Therefore, given some specific parameter space $\mathcal{V}$, we will set $\Sigma$ to guarantee that for any $v \in \mathcal{V}$, $\Sigma^{\frac{1}{2}}v  \in S^{p-1}$ (in fact, we just need 0<$\|\Sigma^{\frac{1}{2}}v\|_2 < \infty$), and  we derive the Gaussian complexity $\tilde{\ell}^*(\mathcal{V})$.
\subsubsection{Sparse group normalization in sub-Gaussian random design}
Consider the parameter space
\begin{equation}
V_{s_0} = \{v:supp(v)\in \mathcal{S}\}\cap \{v:\Sigma^{\frac{1}{2}}v \in S^{p-1}\},
\end{equation}
where $supp(v)$ represents the support set of $v$ and $\mathcal{S}$ is given in \eqref{normset}.
Similarly, we define
$$
\tilde{V}_{s_0}\triangleq \{v:supp(v)\in \mathcal{S}\}\cap \{v:\|\Sigma^{\frac{1}{2}}v\|_2 \le 1\}.
$$
It is easy to observe that $V_{s_0} \subseteq \tilde{V}_{s_0}$. 

\begin{theorem}\label{sgnorm gaussian}
For $1< s_0< \frac{d}{2}$ , it holds that
$$
\ell^*(V_{s_0}) \le 6\left(\log m + s_0\log \frac{5ed}{s_0}\right)^{\frac{1}{2}}.
$$
So that for any set $S\in \mathcal{S}$, assume that the sub-matrix $\Sigma_S$ satisfies that $\lambda_{\max}(\Sigma_{S} )\le \frac{1}{(1+\theta)^2}$. 
If sample size $n$ satisfies that for some constant $0<\theta<1$,
\begin{equation}\label{nsgnorm}
n > \frac{c'\alpha^4}{\theta^2} \left(\log m + s_0\log \frac{ed}{s_0}\right),
\end{equation}
then with probability of at least $1 - \exp(-\bar{c}\theta^2n/\alpha^4)$, we have for all $v \in \{v: supp(v)\in \mathcal{S}\}$, 
$$
\frac{\|Xv\|_n}{\|v\|_2} \le 1,
$$
where $c'$, $\bar{c} > 0$ are absolute constants.
\end{theorem}

\subsubsection{SSGRE condition and WSGRE condition in sub-Gaussian random design}

In this section, we derive the sub-Gaussian random design version of two RE conditions we used for sparse group Lasso and Slope. 

As mentioned above, the calculation of Gaussian complexity is the key point. Owing to the technique of Theorem \ref{gauss2}, we can derive the complexity for this two estimator. The technique is also based on \cite{bellec2018slope}, and we will demonstrate that the condition we require is weaker than classic sub-Gaussian random design result \cite{zhou2009restricted}.

Because we have given the random design property for sparse group normalization, in this part, we assume that sparse group normalization holds. we have the following Theorem to ensure the validity of weighted sparse group RE. And strong sparse group RE holds as follows.

\begin{theorem}[WSGRE for sub-Gaussian random design]\label{randomwsgre}
Let the covariance matrix $\Sigma$ satisfy the RE condition on $\mathcal{C}_{WSGRE}(s,s_0,c_0)$, that is,
\begin{equation}
\min_{v \in \mathcal{C}_{WSGRE}(s,s_0,c_0)\setminus\{0\}}\frac{\|\Sigma^{\frac{1}{2}}v\|_2}{\|v\|_2} = \kappa \in (0,1),
\end{equation}
and the sparse group normalization condition, that is,
$$
\sup_{S\in \mathcal{S}}\|\Sigma_S\|_2\le 1,\ \text{for all}\ S \in \mathcal{S}.
$$
Let the generation of random design matrix $X = Z\Sigma^{\frac{1}{2}} \in  \mathbb{R}^{n\times p}$, if the sample size $n$ satisfies
$$
n > C_1\left(s\log \frac{4em}{s}+ss_0\log\frac{2ed}{s_0}\right),
$$
then with the probability not less than $1-\exp(-C_2n)$, for $\forall v \in \mathcal{C}_{WSGRE}(s,s_0,c_0)\setminus \{0\}$, we have
$$
\frac{\|Xv\|_2}{\|v\|_2} \ge \kappa',
$$
where $C_1,C_2>0$ is an absolute constant, $\kappa \in (0,\kappa)$.
\end{theorem}

\begin{corollary}[SSGRE for sub-Gaussian random design]\label{cor1}
If the covariance matrix satisfies the SSGRE condition, that is, 
$$
\min_{v \in \mathcal{C}_{SSGRE}(s,s_0,c_0)\setminus\{0\}}\frac{\|\Sigma^{\frac{1}{2}}v\|_2}{\|v\|_2} = \kappa \in (0,1).
$$
Assume all the conditions in Theorem \ref{randomwsgre} hold. If the sample size $n$ satisfies
$$
n > C_1\left(s\log \frac{4em}{s}+ss_0\log\frac{2ed}{s_0}\right),
$$
then with the probability greater than $1-\exp(-C_2n)$, for $\forall v \in \mathcal{C}_{SSGRE}(s,s_0,c_0)\setminus \{0\}$, we have
$$
\frac{\|Xv\|_2}{\|v\|_2} \ge \kappa',
$$
where $C_1,C_2>0$ are absolute constants and $\kappa \in (0,\kappa)$.
\end{corollary}

\begin{remark}
The proof of the SSGRE condition here can also be used to prove the DSRE condition. If the method of \cite{zhou2009restricted} is used to prove the DSRE condition, that is, by calculating the covering number of the set $\mathcal{DS}(s,s_0)$, and then by \cite{ledoux1991probability} To find the Gaussian complexity of the generic chaining technique, they use the condition:
$$
\rho_{\max} = \sup_{S \in \mathcal{DS}(s,s_0)}\|\Sigma_{S}\|_2 < \infty.
$$

This assumption is much more stringent than sparse group normalization condition we  use in Theorem \ref{randomwsgre}, because $\mathcal{DS}(s,s_0)$ is much larger than $\mathcal{S}$. This nuance is owing to the theoretical technique of \cite{bellec2018slope}. So that we achieve our goals to clarify all the results in this paper relying solely on  normalization and RE condition. We avoid using incoherence condition like \cite{cai2019sparse} and even RIP-type condition.
\end{remark}

\section{Discussion}\label{discussion}
In this paper, we study the high-dimensional double sparse regression, in which element-wise sparsity and group-wise sparsity exist simultaneously. From the perspective of convex algorithm, we study two kinds of $\ell_1+\ell_{1,2}$ algorithm: the universal tuning parameter for Lasso and weighted tuning parameter for Slope. We derive the estimation upper bounds for both and clarify that they are optimal by deriving  minimax lower bound for double sparsity. Compared with the previous study, we make significant improvements in various aspects:
\begin{itemize}
	\item \cite{cai2019sparse} obtained matching upper and lower bounds for sparse group Lasso. However, we make more efforts on design matrix condition and answer a pending question that how we tackle double sparsity by restricted eigenvalue condition. 
	\item We are inspired by the technique presented in \cite{bellec2018slope}, which establishes the optimality of Slope and opens the possibility of extending it to handle complicated sparsity patterns. While \cite{groupslope} has introduced group Slope based on \cite{bogdan2015slope}, it faces challenges when dealing with double sparsity using the same approach. To address this issue, we propose leveraging the powerful methodology from \cite{bellec2018slope} in conjunction with inequality techniques. This not only allows us to develop Slope for double sparsity but also enables us to derive the minimax optimality on the corresponding restricted eigenvalue condition.
	\item We study the random design property and optimal sample complexity under weak moment distribution and sub-Gaussian distribution. Weak distribution condition is proposed by \cite{lecue2017sparse} and we generalize their results for our sparse group setting. For sub-Gaussian random design, our result is more general than \cite{zhou2009restricted} because we get rid of RIP-type condition. It is owing to our Theorem \ref{gauss2} that we obtain the Gaussian complexity for sparse group Lasso and Slope.
\end{itemize}

Furthermore, there are several important problems that can be addressed in the future. Firstly, an effective algorithm for sparse group SLOPE can be developed, building upon the ideas proposed in our current work.
Secondly, the issue of adaption to the unknown sparsity $s_0$ still remains to be solved. One potential approach to achieve adaption to $s_0$ is through some minimax adaption procedure, which has been widely used and studied in the related field of statistics \cite{lepskii1991problem, lepski1997optimal, dalalyan2022all, aeckerle2022sup}.
Moreover, our approach presents a novel solution to address simultaneous sparse structures, making it a versatile and valuable technique applicable in various domains and applications. For instance, it can be applied to simultaneously sparse and low-rank structure recovery \cite{oymak2015simultaneously, hao2020sparse} as well as sparse tensor SVD \cite{zhang2019optimal}.

\newpage
\begin{appendix}
\section*{Technical Proofs}\label{appn} 
\subsection{Proof of Lemma \ref{phi}}
\begin{proof}
	Let $\Lambda \in [m]$ be the column index sets corresponding to $\Psi_{2}(s)$.
	We consider to upper bound the moment generating function of $\Psi_{2}(s)$.
	For simplicity, let $\sigma^2 = 1$. For $t>0$, we have
	\begin{align}
	\begin{split}\label{chi2}
	\mathbf{E}\left\{\exp\left[t\cdot \Psi_{2}(s)\right]\right\} &= \mathbf{E}\left\{\exp\left[\frac{t}{s s_0}\sum_{j=1}^s\sum_{i=1}^{s_0}(\Phi^*_{i,\Lambda_j})^2\right]\right\}\\
	& \le \frac{1}{s}\sum_{j=1}^s\mathbf{E}\left\{\exp\left[\frac{t}{s_0}\sum_{i=1}^ {s_0}(\Phi^*_{i,\Lambda_j})^2\right]\right\}\\
	&\le \frac{1}{s}\sum_{j=1}^m\mathbf{E}\left\{\exp\left[\frac{t}{s_0}\sum_{i=1}^ {s_0}(\Phi^*_{i,j})^2\right]\right\},
	\end{split}	
\end{align}
	where the first inequality follows from the AM-GM inequality.
	
	Next, we prove the upper bounds of $\mathbf{E}\left\{\exp\left(\frac{t}{s_0}\sum_{i=1}^{s_0}(\Phi^*_{i,j})^2\right )\right\}$.
	Take the $j$-th column as an example. 
	It is easy to observe that
	\begin{align*}
	\exp\left(\frac{t}{s_0}\sum_{i=1}^{s_0}(\Phi^*_{i,j})^2\right) &\leq \sum_{\begin{subarray}{l} S \subseteq [d] \\ |S| = s_0 \end{subarray}}\exp\left(\frac{t}{s_0}\sum_{i \in S}\Phi_{i,j}^2\right)\\
	&\leq \binom{d}{s_0}\exp\left(\frac{t}{s_0}\sum_{i \in S}\Phi_{i,j }^2\right).
	\end{align*}
	Taking the expectation on both sides, we can obtain
	\begin{equation}\label{chi3}
	\mathbf{E}\left\{\exp\left(\frac{t}{s_0}\sum_{i=1}^{s_0}(\Phi^*_{i,j})^2\right)\right\}\le \binom{d}{s_0}\mathbf{E}\left\{\exp\left(\frac{t}{s_0}\sum_{i=1}^{s_0}\Phi_{i,j }^2\right)\right\}.
	\end{equation}
	Let $Z\sim \chi^2_{s_0}$. Then the moment generating function of $Z$ is given by
	$$
	M_Z(t) = (1-2t)^{-\frac{s_0}{2}},\ t < \frac{1}{2}.
	$$
	Therefore, according to the condition \eqref{sgnorm}, we have
	\begin{equation}\label{chi4}
	\begin{aligned}
	\mathbf{E}\left\{\exp\left[\frac{t}{s_0}\sum_{i=1}^{s_0}\Phi_{i,j}^2\right]\right\} & \le \mathbf{E}\left\{\exp\left[\frac{t}{s_0}Z\right]\right\}\\
	&= (1-\frac{2t}{s_0})^{-\frac{s_0}{2}}.
	\end{aligned}
	\end{equation}	
    Combining \eqref{chi2}-\eqref{chi4}, according to Chernoff bound, we have
	\begin{align*}
	\mathbf{P}(\Psi_{2}(s)\ge a) &\leq \inf_{t\ge0}\frac{ \mathbf{E}\left\{\exp\left(t\cdot\Psi_{ 2}(s)\right)\right\}}{e^{at}}\\
	&\le \frac{m}{s}\binom{d}{s_0}2^{s_0}\exp(-\frac{3 }{8}as_0)\\
	&\le \frac{4m}{s}\exp\left\{-\frac{3}{8}s_0\left(a-\frac{8}{3}\log\frac{2ed}{s_0} \right)\right\},
	\end{align*}
    where the second inequality follows from $t = \frac{3}{8}s_0$.
	Let $a = \frac{8}{3}(\log\frac{2ed}{s_0}+ \frac{2}{s_0}\log \frac{4m}{s})$, we obtain
	\begin{equation*}
	\mathbf{P}\left[\sup_{S\in \mathbb{S}_2(s,s_0)}\frac{1}{ss_0\sigma^2}\|\Phi_S\|_2^2 \ge \frac{8}{3}\left(\log\frac{2ed}{s_0}+ \frac{2}{s_0}\log \frac{4m}{s}\right)\right] \le \frac{ s}{4m}.
	\end{equation*}
	
	Next, we turn to the analysis of the random variable $\Psi_{1}(s)$ on $\mathbb{S}_1(s,s_0)$. By the previous definition, $\Psi_{1}(s)\ge\Psi_{2}(s)$ holds almost surely. 
	However, we prove that the tail probability inequalities of the two variables only differ by a constant factor.
	
	Here we consider a set family $\mathcal{F}$ consisting of set $F \in \mathbb{S}_2(1, s_0)$.
	In particular, given any column, we select any $s_0$ elements of this column to form an index set $F$.
	The cardinality of the obtained set family is $\mathcal{F} = m \binom{d}{s_0}$.

	Based on columns, the index set  $S \in \mathbb{S}_2(s, s_0)$ can be separated into $s$ subsets $S^j,\ j \in [s]$. 
	Consequently, we can find $s$ elements of $\mathcal{F}$ such that $S^j \subseteq F^j, j \in [s]$ to cover $S$.
	However, it is not sufficient to find $s$ elements of $\mathcal{F}$ to cover any index set $S \in \mathbb{S}_1(s,s_0)$. 
    In what follows, we show that $2s$ sets of $\mathcal{F}$ are sufficient to cover any $S \in \mathbb{S}_1(s, s_0)$.
	For any $S \in \mathbb{S}_1(s, s_0)$, let $a_j$ be the number of elements in the $j$-th column of $S$, 
	, where we only need to consider $s$ non-zero columns. Observe that
	$$
	a_1+a_2+,\ldots,+a_s = ss_0.
	$$
	For $S^j,\ j \in [s]$, we can use at most $\lfloor\frac{a_i}{ s_0}\rfloor+1$ sets of $\mathcal{F}$ to cover $S^j$.
	Then according to the properties of the floor function, that is,  $\lfloor x\rfloor+\lfloor y\rfloor\le\lfloor x+y\rfloor$, we obtain
	$$
	\sum_{i=1}^s\left\{\lfloor\frac{a_i}{s_0}\rfloor+1\right\} \le s+\lfloor\frac{\sum_{i=1}^ sa_i}{s_0}\rfloor = 2s.
	$$
	Therefore, we can divide $S$ into $2s$ disjoint sets of set family $\mathcal{F}$. 
	Similar to \eqref{chi2}, we have
	\begin{align*}
	\mathbf{E}\left\{\exp\left(t\cdot \Psi_{1}(s)\right)\right\} &= \mathbf{E}\left\{\exp\left(\frac {1}{2s}\sum_{j=1}^{2s}\frac{2t}{s_0}\|\Phi_{S^j}\|^2_2\right)\right\}\\
	& \le \frac{1}{2s}\sum_{F \in \mathcal{F}}\mathbf{E}\left\{\exp\left(\frac{2t}{s_0} \|\Phi_{U}\|^2_2\right)\right\}.
	\end{align*}
	The rest of the proof is similar to the case of $\mathbb{S}_2(s,s_0)$. We obtain
	
	$$
	\mathbf{P}\left[\sup_{S\in \mathbb{S}_1(s,s_0)}\frac{1}{ss_0\sigma^2}\|\Phi_S\|_2^2 \ge \frac{16}{3}\left(\log\frac{2ed}{s_0}+ \frac{2}{s_0}\log \frac{4m}{s}\right)\right] \le \frac{ s}{4m}.
	$$
	Without loss of generality, we ignore the difference of constant factor, and rewrite the conclusion of $\mathbb{S}_2(s,s_0)$ as:
	$$
	\mathbf{P}\left[\sup_{S\in \mathbb{S}_2(s,s_0)}\frac{1}{ss_0\sigma^2}\|\Phi_S\|_2^2 \ge \frac{16}{3}\left(\log\frac{2ed}{s_0}+ \frac{2}{s_0}\log \frac{4m}{s}\right)\right] \le \frac{ s}{4m}.
	$$
	These complete the proofs of Lemma \ref{phi}.
	\end{proof}

\subsection{Proof of Lemma \ref{lem42}}

\begin{proof}
	Following the definitions, we can easily obtain that
	$$
	\Upsilon_s \le \sqrt{s_0\Psi_{2}(s)}\quad \text{and}\quad \upsilon_s \le \sqrt{\Psi_{1}(s)}.
	$$
	The following is proved using the peeling technique \citep{geer2000empirical}. With respect to $\Upsilon_{s}$, according to \eqref{s2}, for any given $j \in [m]$, we have
	\begin{equation}\label{lem421}
	\mathbf{P}\left(\Upsilon_j \le \frac{4}{\sqrt{3}}\sqrt{s_0}\lambda_j\right)\ge 1-\frac{s}{4m}, \qquad j \in [p].
	\end{equation}
	Let $q$ be the integer such that $2^q\le m\le 2^{q+1}$.
    Apply \eqref{lem421} to $j = 2^{l}$ for $l=0,1,2\cdots,q-1$ and define the event $\Omega_1$ defined as
	$$
	\Omega_1 = \left\{\max_{l=0,1,\cdots,q-1}\frac{\Upsilon_{2^l}\sqrt{3}}{4\sqrt{s_0}\lambda_{2^l}}\le 1\right\}.
	$$
	By the union bound, we have $P(\Omega_1) \ge 1- \frac{\sum_{l=0}^{q-1}2^l}{4m}\geq 1- \frac{2^q-1}{4m} \ge \frac{3}{4}$.
    For any $j < 2 ^q$, there exists $l \in \{0, 1, \ldots, q-1\}$ such that $2^l < j < 2^{l+1}$, and, thus, with the satisfaction of event $\Omega_1$,
	\begin{align*}
		\Upsilon_{j} \le \Upsilon_{2^l}
		&\le \frac{4\sigma}{\sqrt{3}}\sqrt{s_0\log\frac{2ed}{s_0}+ 2\log \frac{4m}{2^{l}}}\\
		&<4\sigma\sqrt{\frac{s_0}{3}\log\frac{2ed}{s_0}+ \frac{2}{3}\log \frac{8m}{j}}\\
		&<4\sigma\sqrt{s_0\log\frac{2ed}{s_0}+ 2\log \frac{2m}{j}}\\
		&< 4\sqrt{s_0}\lambda_j.
	\end{align*}
	On the other hand, for $2^{q}<j\le m\le 2^{q+1}$, there are
	\begin{align*}
		\Upsilon_{j} \le \Upsilon_{2^{q-1}}
		&\le \frac{4\sigma}{\sqrt{3}}\sqrt{s_0\log\frac{2ed}{s_0}+ 2\log \frac{4m}{2^{q-1}}}\\
		&<4\sigma\sqrt{\frac{s_0}{3}\log\frac{2ed}{s_0}+ \frac{2}{3}\log \frac{16m}{2^{q+1}}}\\
		&<4\sigma\sqrt{\frac{s_0}{3}\log\frac{2ed}{s_0}+ \frac{2}{3}\log \frac{16m}{j}}\\
		&<4\sigma\sqrt{s_0\log\frac{2ed}{s_0}+ 2\log \frac{4m}{j}}\\
		&= 4\sqrt{s_0}\lambda_j.
	\end{align*}
	Thus, when event $\Omega_1$ holds, we have
	$
	\max_{j=0,1,\cdots,m}\Upsilon_j\le 4\sqrt{s_0}\lambda_j
	$
	for all $j = 0, 1, \ldots, m$.
	
	Next, we turn to the proof of $\upsilon_s$.
	We first show the monotonicity of $\upsilon_s$, i.e., $\upsilon_s \ge \upsilon_{s+1}$ for all $s \in [m-1]$.
	By the definition of $\upsilon_{s+1}$, there exists $s+1$ columns such that the $s_0\times (s+1)$-th largest absolute value among these columns is $\upsilon_{s+1}$. 
	We denote the column that the $\upsilon_{s+1}$ belongs to as $g$.
	We separate our discussion into two cases:
	\begin{enumerate}
		\item If column $g$ contains fewer than $s_0$ elements that are larger than $\upsilon_{s+1}$, the remaining $s$ columns contain more than $ss_0$ elements that are larger than $\upsilon_{s+1}$. Therefore, among these $s$ columns, the $ss_0$-th largest element in absolute value is greater than $\upsilon_{s+1}$. Conversely, by definition, it is smaller than $\upsilon_s$. Consequently, we can conclude that $\upsilon_s \geq \upsilon_{s+1}$.		
		\item If column $g$ has at least $s_0$ elements larger than $\upsilon_{s+1}$, then among the remaining $s$ columns, there exists $s-1$ columns such that the combined $s$ columns (including column $g$) contain at least $ss_0$ elements greater than $\upsilon_{s+1}$. Therefore, within these $s$ columns, the $ss_0$th largest element in absolute value is greater than $\upsilon_{s+1}$ and, by definition, it is smaller than $\upsilon_s$. Consequently, we have $\upsilon_s \geq \upsilon_{s+1}$.		
	\end{enumerate}
	The rest of the work is just a matter of combining \eqref{s1} with the monotonicity of $\upsilon_s$. Similarly, we can obtain event
	$$
	\Omega_3 = \left\{\max_{l=0,1,\cdots,q-1}\frac{\upsilon_{2^l}\sqrt{3}}{4\lambda_{2^l}}\le 1\right\}
	$$
	satisfying $P(\Omega_3) \ge \frac{3}{4}$. By the union bound, we have 
	$$
	\mathbf{P}(\Omega_1\cap \Omega_3) \ge \frac{3}{4}+\frac{3}{4}-1 = \frac{1}{2}. 
	$$
    This completes the proof of Lemma \ref{lem42}.
\end{proof}

\subsection{Proof of Theorem \ref{gauss2}}
Before proving Theorem \ref{gauss2}, we provide the rearrangement inequality used frequently in the following proof.
In specific, the rearrangement inequality states that
for every choice of real numbers
$$
x_1 \leq \cdots \leq x_n \quad \text { and } \quad y_1 \leq \cdots \leq y_n
$$
and every permutation
$
x_{\sigma(1)}, \ldots, x_{\sigma(n)}
$
of $x_1, \ldots, x_n$,
it holds that
$$
x_n y_1+\cdots+x_1 y_n \leq x_{\sigma(1)} y_1+\cdots+x_{\sigma(n)} y_n \leq x_1 y_1+\cdots+x_n y_n.
$$

\begin{proof}
    We separate the proof of Theorem \ref{gauss2} into four steps.

	{\bf Step 1}:
	According to the definition of the $U,\Phi$ matrix in \ref{notation}, \eqref{gauss} can be expressed as
	$$
	\left|\frac{1}{n}\xi^TXu\right| = \left|\frac{1}{\sqrt{n}}\sum_{j=1}^m\sum_{i=1}^d U_{i,j}\Phi_{i,j}\right|.
	$$
	By the rearrangement inequality, we have
	$$
	\left|\frac{1}{\sqrt{n}}\sum_{j=1}^m\sum_{i=1}^d U_{i,j}\Phi_{i,j}\right| \le \frac{1}{\sqrt{n}}\sum_{j=1}^m\sum_{i=1}^d U^*_{i,j}\Phi^*_{i,j}.
	$$
	
	{\bf Step 2}:
	Divide $U^*,\Phi^*$ into two parts by rows, i.e., write
	$$
	\frac{1}{\sqrt{n}}\sum_{j=1}^m\sum_{i=1}^d U^*_{i,j}\Phi^*_{i,j} = \underbrace{\frac{1}{\sqrt{n}}\sum_{j=1}^m\sum_{i=1}^{s_0} U^*_{i,j }\Phi^*_{i,j}}_{A_1}+\underbrace{\frac{1}{\sqrt{n}}\sum_{j=1}^m\sum_{i=s_0}^{d} U^*_{i,j}\Phi^*_{i,j}}_{A_2}.
	$$
	
	For part $A_1$, we define the subvectors consisting of the first $s_0$ elements of each column of $U^*,\Phi^*$ as $U^{*}_{A_1,j},\Phi^{*}_{A_1,j}$, respectively. Then, by the Cauchy-Schwarz inequality, we have
	
	\begin{align*}\label{gauss21}
		\begin{split}
		A_1&\le \frac{1}{\sqrt{n}}\sum_{j=1}^m \|U^{*}_{A_1,j}\|_2\cdot \|\Phi^{*}_{A_1,j}\|_2\\\
		&\le \frac{1}{\sqrt{n}}\sum_{j=1}^m \|U^*_{j}\|_2\cdot \|\Phi^{*}_{A_1,j}\|_2\\
        &\le \frac{1}{\sqrt{n}}\sum_{j=1}^m \|U_{j*}\|_2\cdot \|\Phi^*_{A_1,j^*}\|_2,
			\end{split}
	\end{align*}
	where the last inequality follows the rearrangement inequality.
	Note that $\Phi^*_{A_1,j^*}$ is  $\Upsilon_j$ as defined in Lemma \ref{lem42}.
	Therefore, we have
	\begin{equation}\label{lemma42I}
		\begin{aligned}
			A_1 &\le \frac{1}{\sqrt{n}}\sum_{j=1}^m\|U_{j^*}\|_2\cdot\Upsilon_j\\
			&\le \frac{1}{\sqrt{n}}\sum_{j=1}^m\|U_{j^*}\|_2\cdot\sqrt{s_0}\lambda_j\frac{\Upsilon_j}{\sqrt{s_0}\lambda_j}\\\
			&\le \frac{1}{\sqrt{n}}\sum_{j=1}^m\|U_{j^*}\|_2\cdot\sqrt{s_0}\lambda_j\left(\max_{j\in [m]}\frac{\Upsilon_j}{\sqrt{s_ 0}\lambda_j}\right).
		\end{aligned}
	\end{equation}
	
	{\bf Step 3}.
	For part $A_2$, we first divide the index set of part $A_2$ into some disjoint subsets $S_1,S_2,\cdots$, where each subset is of size $s_0$ (here we assume that $d$ is larger than $2s_ 0$, otherwise we can fill in with zeros).
    Then, we use the random variable $\upsilon_s$ defined in Lemma \ref{lem42}.
	As a result, we have
	$$
	A_2 = \frac{1}{\sqrt{n}}\sum_{q}\sum_{(i,j)\in S_q}U^*_{i,j}\Phi^*_{i,j}.
	$$
	Without loss of generality, we assume that the absolute values of the $s+1$-th row of $\Phi^*$ is rearranged in descending order from left to right by columns. 
	Next, we specify the construction the sets $\{S_q\}$:
	\begin{enumerate}
		\item Select the largest $s_0$ elements from the first column of part $A_2$ to form set $S_1$;
		\item Consider the first $q$ ($q\geq2$) columns of part $A_2$, and select the largest $s_0$ elements from the portion excluding set $\cup_{i=1}^{q-1}{S_i}$ to form set $S_q$.
		\item Repeat step 2 $m$ times to obtain $m$ sets of size $s_0$. Allocate the remaining elements arbitrarily into sets of size $s_0$, and finally, obtain the family of sets $\{S_q\}$.
	\end{enumerate}
	
	We denote the maximum value of $\Phi^*$ within set $S_q$ by $\|\Phi^*_{S_q}\|_{\infty}$. Next, we prove that \textbf{for any given $q \in [m]$, there exists at least $s_0\times q$ elements in the first $q$ columns of $\Phi^*$ that are no less than $\|\Phi^*_{S_q}\|_{\infty}$}.

	\begin{remark}
		The above claim is combined with $\upsilon_q$ defined in Lemma \ref{lem42}. The definition of $\upsilon_q$ here still uses the whole random matrix $\Phi$ in Lemma \ref{lem42}, not the $A_2$ part. And $\Phi^*$ only permutes the order of the elements in the same column of $\Phi$, so it does not affect the definition of $\upsilon_q$. Therefore, if the claim is clarified, we have $\|\Phi^*_{S_q}\|_{\infty} \le \upsilon_q$ for $q\in [m]$ by the definition of $\upsilon_q$. For $q>m$,
we have $	\|\Phi^*_{S_m}\|_{\infty} \geq \|\Phi^*_{S_q}\|_{\infty}$ for any $q \geq m+1$  by the construction of $\{S_q\}$.
	\end{remark}

We prove the above claim  using mathematical induction:
	\begin{enumerate}
		\item When $q=1$, we have $\|\Phi^*_{S_1}\|_{\infty} = \Phi^*_{s_0+1,1}$. Then we have
		$$
		\Phi^*_{s_0+1,1} \le \Phi^*_{s_0,1},
		$$
		which implies that the first $s_0$ elements of $\Phi^*_1$ are no less than $\|\Phi^*_{S_1}\|_{\infty}$.
		Therefore, we can conclude that there are $s_0$ elements 
 no less than $\|\Phi^*_{S_q}\|_{\infty}$ in the first column.
		\item For $q < m$, we assume that the claim holds. 
		Note that there are at least $q\cdot s_0$ elements larger than $\|\Phi^*_{S_{q}}\|_{\infty}$ in the first $q$ columns by assumption. 
		For $q+1$, 
		we can analyze the location of $\|\Phi^*_{S_{q+1}}\|_{\infty}$ by the following two cases:
		\begin{enumerate}
			\item $\|\Phi^*_{S_{q+1}}\|_{\infty}$ appears in the first $q$ columns of part $A_2$. 
			By the construction of $S_{q+1}$, we have that all the elements of $\Phi^*_{S_q}$ are larger than $\|\Phi^*_{S_{q+1}}\|_{\infty}$.
			Since $|S_q|=s_0$ and $\|\Phi^*_{S_{q}}\|_{\infty}\geq \|\Phi^*_{S_{q+1}}\|_{\infty}$, we have that there are at least $q\cdot s_0+s_0$ elements are larger than $\|\Phi^*_{S_{q+1}}\|_{\infty}$ in the first $q$ columns.
			Therefore,  we have there are no less than $s_0\times(q+1)$ elements that larger than
			$
			\|\Phi^*_{S_{s+1}}\|_{\infty}
			$.
			
			\item $\|\Phi^*_{S_{q+1}}\|_{\infty}$ appears in the $(q+1)$-th column.
			In this case, we must have $\|\Phi^*_{S_{q+1}}\|_{\infty} = \Phi^*_{s_0+1,q+1}$.
			Observe that 
			$$
			\Phi^*_{s_0+1,q+1} \le \Phi^*_{s_0+1,q'} \le \Phi^*_{s_0,q'},\quad \forall q' \le q+1.
			$$
			Thus we get that there are at least $(q+1)\cdot s_0$ elements are larger than $\|\Phi^*_{S_{q+1}}\|_{\infty}$ in the first $q+1$ columns.
		\end{enumerate}
	\end{enumerate}
	Overall, we have proved that $\|\Phi^*_{S_q}\|_{\infty} \le \upsilon_q$ for $q \in [m]$. For $q> m$, we obviously have $\|\Phi^*_{S_q}\|_{\infty}\le \upsilon_m$. Arrange all elements of $u$ in the $A_2$ part in descending order of absolute value, and the $i$-th position is noted as $u^*_{A_2,i}$.
	\begin{equation}\label{lemma42II}
		\begin{aligned}
			A_2 & \le \frac{1}{\sqrt{n}}\sum_{q}\sum_{(i,j)\in S_q} U^*_{i,j} \upsilon_{q \wedge m}\\
			& \le \frac{1}{\sqrt{n}}\sum_{q}\sum_{(i,j)\in S_q} U^*_{i,j} \lambda_{q \wedge m} \left(\max_{j\in [m]}\frac{\upsilon_j}{\lambda _j}\right)\\
			& \le \frac{1}{\sqrt{n}}\sum_{i=1}^{m\times (d-s_0)}u^*_{A_2,i}\tilde{\lambda}_{i}\left(\max_{j\in [m]}\frac{\upsilon_j}{\lambda _j}\right)\\
			&\le \frac{1}{\sqrt{n}}\sum_{i=1}^{p}u_{i^*}\tilde{\lambda}_{i}\left(\max_{j\in [m]}\frac{\upsilon_j}{\lambda_j}\right),
		\end{aligned}
	\end{equation}
	where the third inequality follows from the rearrangement inequality.

	{\bf Step 4}:
	When event $\Omega$ defined by Lemma \ref{lem42} satifies that $\mathbf{P}(\Omega)\ge \frac{1}{2}$, \eqref{lemma42I} and \eqref{lemma42II} together show that
	\begin{equation}
		\begin{aligned}
			\sup_{u:N(u)\le 1}\left|\frac{1}{n}\xi^TXu\right| &\le A_1+A_2\\
			&\le\frac{1}{\sqrt{n}}\sum_{j=1}^m\|U_{j^*}\|_2\sqrt{s_0}\lambda_j\left(\max_{j\in[m]}\frac{\Upsilon_j}{\sqrt{s_0}\lambda_j}\right)\\\
			&+\frac{1}{\sqrt{n}}\sum_{i=1}^{p}u_{i^*}\tilde{\lambda}_{i}\left(\max_{j\in[m]}\frac{\upsilon_j}{\lambda_j}\right)\\
			&\le N(u)\cdot \left(\max_{j\in [m]}\frac{\Upsilon_j}{\sqrt{s_0}\lambda_j}\right)\vee\left(\max_{j\in[m]}\frac{\upsilon_ j}{\lambda_j}\right)=4.
		\end{aligned}
	\end{equation}
\end{proof}

\subsection{Proof of Theorem \ref{gauss3}}
The proof of Theorem \ref{gauss3} can be found in Proposition E.3 of \cite{bellec2018slope}, which is added here for the completeness of the article. The core idea of the proof of this theorem is the use of Levy-type inequalities for Lipschitz functions on sub-Gaussian random variables, as specified in \cite{boucheron2013concentration}, Theorem 10.17 or \cite{ledoux1991probability}, Eq. (1.4).
\begin{proof}
	By homogeneity, it is enough to consider the set
	$$
	T\triangleq \left\{\forall u\in \mathbb{R}^p:\max\left(N(u),\frac{1}{L}\|Xu\|_n\right)\le 1\right\},
	$$
	where $L\triangleq (\frac{n}{\sigma^2\log(1/\delta_0)})^{\frac{1}{2}}$.
	Define function $f : \mathbb{R}^n \rightarrow \mathbb{R}$ by
	$$
	f(t) \triangleq \sup_{u\in T}\frac{1}{n}(\sigma t)^{\top}Xu\qquad \text{for all}\ t \in \mathbb{R}^n.
	$$
	It is obvious that $f(t)$ is a Lipschitz function with Lipschitz constant $\frac{\sigma L}{\sqrt{n}}$.
	Thus by \cite{boucheron2013concentration}, with probability at least $1-\delta_0/2$, we have
	$$
	\begin{aligned}
		\sup_{u\in T}\frac{1}{n}\xi^{\top}Xu&\le \text{Med}\left[\sup_{u\in T}\frac{1}{n}\xi^{\top}Xu\right] + \sigma L\sqrt{\frac{2\log (1/\delta_0)}{n }}\\\
		& \le \text{Med}\left[\sup_{u:N(u)\le 1}\frac{1}{n}\xi^{\top}Xu\right] + \sigma L\sqrt{\frac{2\log (1/\delta_0)}{n}}\\\
		& \le 4 + \sigma L\sqrt{\frac{2\log (1/\delta_0)}{n}} = 4+\sqrt{2},
	\end{aligned}
	$$
	where the last inequality uses $P(\Omega_2) \geq \frac{1}{2}$to bound the median.
\end{proof}

\section{Proof of upper bounds for two estimators}
\subsection{Proof of Theorem \ref{sglassomain}}
\begin{proof}
	According to the convex optimization problem \eqref{sglasso} and Lemma \ref{lemma:sglasso}, we have
	$$
	\|X(\hat\beta-\beta^*)\|_n^2 \le \underbrace{\frac{1}{n}\xi^TX(\hat\beta-\beta^*)}_{B_1}+ \underbrace{\frac{1}{2}(h(\beta)-h(\hat{\beta^*}) )}_{B_2}.
	$$ 
	For part $B_1$, recall the definition of $N(u)$ in \eqref{Nu}.
	By the Cauchy-Schwarz inequality, we have
	\begin{align}
		\begin{split}\label{lassomain2}
		\sqrt{n}N(u)\le &\left(\sum_{i=1}^{ss_0}\tilde{\lambda}_i^2\right)^{\frac{1}{2}}\|u\|_2+\sum_{i=ss_0+1}^p\tilde{\lambda}_iu_{i^*}+\left(\sum_{j=1}^{s}\lambda_j^2\right)^{\frac{1}{2}}\sqrt{s_0}\|u\|_2+\\
		&\sum_{j=s+1}^m\sqrt{s_0}\lambda_j\|U_{j^*}\|_2 .
		\end{split}
	\end{align}
	Observe that
	\begin{align}
		\begin{split}\label{Lambda}
		\left(\sum_{i=1}^{ss_0}\tilde{\lambda}_i^2\right)^{\frac{1}{2}} = \sqrt{s_0}\left(\sum_{j=1}^{s}\lambda_j^2\right)^{\frac{1}{2}} =&\sqrt{s_0}\sigma\left(\sum_{j=1}^s\log\frac{2ed}{s_0}+ \frac{2}{s_0}\log \frac{4em}{j}\right)^{\frac{1}{2}}\\
        \leq&C \frac{\gamma\sqrt{nss_0}}{4+\sqrt{2}}\lambda^{\#},
        \end{split}
	\end{align}
	where constant $C \in (1, 2)$. Since the constant term can be ignored, we set $C$ to 1 for convenience.
	On the other hand, $\tilde{\lambda}_{ss_0+1} = \lambda_{s+1}\le \frac{\gamma\sqrt{n}}{4+\sqrt{2}}\lambda^{\#}$.
	Combining this inequality, \eqref{lassomain2} and \eqref{Lambda}, we have
	\begin{align*}
	(4+\sqrt{2})N(u)&\le \gamma \lambda^{\#}\left\{2\sqrt{ss_0}\|u\|_2 + \left(\sum_{i=ss_0+1}^pu_{i^*}+\sum_{j=s+1}^m\sqrt{s_0}\|U_{j^*}\|_2\right)\right\}\\
	&\triangleq F(u).
	\end{align*}	
	By the definition of $\delta(\lambda)$ and $\lambda^{\#}$, we have
	$$
	G(u) = (4+\sqrt{2})\|Xu\|_n\sigma\sqrt{\frac{\log \frac{1}{\delta_0}}{n}} = \lambda^{\#}\gamma\sqrt{ss_0}\left(\frac{\log(1/\delta_0)}{\log( 1/\delta(\lambda^{\#}))ss_0}\right)^{\frac{1}{2}}\|Xu\|_n.
	$$
	From Theorem \ref{gauss3}, with probability at least $1-\frac{\delta_0}{2}$, we have 
	\begin{align}
		\begin{split}\label{lassomain3}
	B_1 &\leq (4+\sqrt{2})\max\left(N(u),\|Xu\|_n\sigma\sqrt{\frac{\log \frac{1}{\delta_0}}{n}}\right)\\
	&\leq \max (F(u), G(u)).
		\end{split}
	\end{align}
	Next, for part $B_2$, it is worth noting that sequences with all entries equal to $\lambda^{\#}$ constitute a special case of non-increasing sequences $\lambda$ and $\tilde{\lambda}$.	
	By Lemma \ref{lemma:sglasso2}, we have
	\begin{align}
		\begin{split}\label{lassomain1}
		B_2&\le \lambda^{\#}\left\{(\|\beta^*\|_1 - \|\hat\beta\|_1) + \sqrt{s_0}(\|\beta^*\|_{1,2}-\|\hat\beta\|_{1,2})\right\}\\
		&\le \lambda^{\#}\left\{\left(\sqrt{ss_0}\|u\|_2-\sum_{i=ss_0+1}^pu_{i^*}\right)+\left(\sqrt{ss_0}\|u\|_2-\sum_{j=s+1}^m\sqrt{s_0}\|U_{j^*}\|_2\right)\right\}\\
		&\le \lambda^{\#}\left\{2\sqrt{ss_0}\|u\|_2 - \left(\sum_{i=ss_0+1}^pu_{i^*}+\sum_{j=s+1}^m\sqrt{s_0}\|U_{j^*}\|_2\right)\right\} .
		\end{split}
	\end{align}
	 Then, combining parts \eqref{lassomain3} and \eqref{lassomain1}, we have
	\begin{align}
		\begin{split}\label{thm1}
		\|Xu\|_n^2 \le& \max\left(F(u),G(u)\right)+\\
		&\lambda^{\#}\left[2\sqrt{ss_0}\|u\|_2 - \left(\sum_{i=ss_0+1}^pu_{i^*}+\sum_{j=s+1}^m\sqrt{s_0}\|U_{j^*}\|_2\right)\right].
		\end{split}
	\end{align}
	The following proof is discussed in two cases:
	\begin{enumerate}
		\item $G(u) \ge F(u)$. In this case, we have
		\begin{equation}\label{lassomain4}
			\|u\|_2 \le \frac{1}{2}\left(\frac{\log(1/\delta_0)}{\log(1/\delta(\lambda^{\#}))ss_0}\right)^{\frac{1}{2}}\|Xu\|_n.
		\end{equation}
		Substituting \eqref{lassomain4} into \eqref{thm1}, we have
		\begin{equation}\label{lassomain5}
		\|Xu\|_n^2 \le (1+\gamma)\lambda^{\#}\sqrt{ss_0}\left(\frac{\log(1/\delta_0)}{\log(1/\delta(\lambda^{\#}))ss_0}\right)^{\frac{1}{2}} \|Xu\|_n.
		\end{equation}
		Combining \eqref{lassomain4} and \eqref{lassomain5}, we have
		\begin{equation*}
			\|u\|_2\le\frac{1}{2}(1+\gamma) \left(\frac{\log(1/\delta_0)}{\log(1/\delta(\lambda^{\#}))ss_0}\right)\sqrt{ss_0}\lambda^{\#}.
		\end{equation*}
		\item $F(u) > G(u)$. Under this assumption, \eqref{thm1} shows that
		\begin{equation}\label{lassomain6}
		\|Xu\|_n^2 \le \lambda^{\#}\left[2(1+\gamma)\sqrt{ss_0}\|u\|_2 - (1-\gamma)\left(\sum_{i=ss_0+1}^pu_{i^*}+\sum_{j=s+1}^m\sqrt{s_0}\|U_ {j^*}\|_2\right)\right] .
		\end{equation}
		Therefore, we have
		$$
		\frac{2(1+\gamma)}{1-\gamma}\sqrt{ss_0}\|u\|_2 > \sum_{i=ss_0+1}^pu_{i^*}+\sum_{j=s+1}^m\sqrt{s_0}\|U_ {j^*}\|_2,
		$$
		which implies that $u$ belongs to cone $\mathcal{C}_{SSGRE}(s, s_0, \frac{4\gamma}{1-\gamma})$.
		Then, according to Assumption \ref{ssgre} with parameter $\theta(s, s_0, \frac{4\gamma}{1-\gamma})$, we have
		$$
		\|u\|_2^2 \le \frac{\|Xu\|_n^2}{\theta(s,s_0,\frac{4\gamma}{1-\gamma})^2}.
		$$
		Substituting \eqref{lassomain6} into above inequality, we have
		$$
		\|u\|_2 \le \frac{2(1+\gamma)\sqrt{ss_0}\lambda^{\#}}{\theta(s,s_0,\frac{4\gamma}{1-\gamma})^2}.
		$$
	\end{enumerate}
	Thus combining these two cases, we conclude that
	$$
	\|\hat{\beta}-\beta^*\|_2 \le (1+\gamma)\sqrt{ss_0}\lambda^{\#}\max\left\{\frac {1}{2}\left(\frac{\log(1/\delta_0)}{\log(1/\delta(\lambda^{\#}))ss_0}\right), \frac{2}{\theta(s,s_0,\frac{4\gamma}{1-\gamma})^2}\right\}.
	$$
\end{proof}

\subsection{Proof of Theorem \ref{sgSlopemain}}
\begin{proof}
	Define $\Lambda(s)$ as
	$$
	\Lambda(s) = \frac{(4+\sqrt{2})}{\sqrt{n}\gamma}\left(\sum_{i=1}^{ss_0}\tilde{\lambda}_i^2\right)^{\frac{1}{2}} = \frac{(4+\sqrt{2})}{ \sqrt{n}\gamma}\sqrt{s_0}\left(\sum_{j=1}^{s}\lambda_j^2\right)^{\frac{1}{2}}.
	$$
	By the definition of $h(\beta)$ and Lemma \ref{lemma:sglasso2}, we have
	\begin{align}
	\begin{split}\label{slope1}
	\frac{1}{2}\left(h(\beta^*)-h(\hat{\beta}) \right)&=\frac{4+\sqrt{2}}{ \sqrt{n}\gamma}\left\{\left(\|\beta^*\|_{\tilde \lambda^*}-\|\hat\beta\|_{\tilde \lambda^*}\right)+\left(\|\beta^*\|_{G, \lambda^*}-\|\hat\beta\|_{G, \lambda^*}\right)\right\}\\
	&\le 2\Lambda(s)\|u\|_2-\frac{4+\sqrt{2}}{\sqrt{n}\gamma}\left(\sum_{i=ss_0+1}^p\tilde{\lambda}_iu_{i^*}+\sqrt{s_0}\sum_{j=s+1}^m\lambda_j\|U_{j^*}\|_2\right).
	\end{split}	
    \end{align}
	Define $K(u) = \frac{4+\sqrt{2}}{\sqrt{n}\gamma}\left(\sum_{i=ss_0+1}^p\tilde{\lambda}_iu_{i^*}+\sqrt{s_0}\sum_{j=s+1} ^m\lambda_j\|U_{j^*}\|_2\right)$, then we have
	$$
	\frac{1}{2}\left(h(\beta^*)-h(\hat{\beta}) \right)\le 2\Lambda(s)\|u\|_2-K(u).
	$$
	According to  \eqref{Nu}, we have
	$$
	(4+\sqrt{2})N(u) \le \gamma\left\{2\Lambda(s)\|u\|_2+ K(u)\right\}.
	$$
	Define $H(u) = \gamma\left\{2\Lambda(s)\|u\|_2+ K(u)\right\}$. By Lemma \ref{lemma:sglasso}, similar to \eqref{thm1}, we get
	\begin{equation}\label{thm2}
		\|Xu\|_n^2 \le \max\{H(u),G(u)\} + 2\Lambda(s)\|u\|_2- K(u),
	\end{equation}
	where $G(u)$ is defined in the proof of Theorem \ref{sglassomain}.
	Similarly, we discuss the same in two cases:
	\begin{enumerate}
		\item $H(u) \le G(u)$. This case implies that
		$$
		2\Lambda(s)\|u\|_2\le\lambda^{\#}\left(\frac{\log(1/\delta_0)}{\log(1/\delta(\lambda^{\#}))}\right)^{\frac{1}{2}}\|Xu\|_n.
		$$
		By the definition of $\lambda^{\#}$, we have
		\begin{equation}\label{Slope3}
			\|u\|_2\le \frac{1}{2}\left(\frac{\log(1/\delta_0)}{\log(1/\delta(\lambda^{\#}))ss_0}\right)^{\frac{1}{2}}\|Xu\|_n.
		\end{equation}
		Substituting above inequality into \eqref{thm2}, we have
		$$
		\|Xu\|_n^2 \le (1+\gamma)\lambda^{\#}\sqrt{ss_0}\left(\frac{\log(1/\delta_0)}{\log(1/\delta(\lambda^{\#}))ss_0}\right)^{\frac {1}{2}}\|Xu\|_n.
		$$
		Combining these two inequalities we have
		\begin{equation*}
			\|u\|_2\le\frac{1+\gamma}{2} \left(\frac{\log(1/\delta_0)}{\log(1/\delta(\lambda^{\#}))ss_0}\right)\sqrt{ss_0}\lambda^{\#}.
		\end{equation*}
		\item $H(u) > G(u)$. In this case, \eqref{thm2} shows that
		\begin{equation}\label{slope2}
			\|Xu\|_n^2 \le 2(1+\gamma)\Lambda(s)\|u\|_2 - (1-\gamma)K(u),
		\end{equation}
		which implies that
		$$
		K(u) < \frac{2(1+\gamma)}{1-\gamma}\Lambda(s)\|u\|_2.
		$$
        Consequently,
		$$
		\begin{aligned}
			\|u\|_{\tilde{\lambda}*}+\sqrt{s_0}\|u\|_{G,\lambda*}&= \sum_{j=1}^s\|U_{j^*}\|\sqrt{s_0}\lambda_j+\sum_{j=1}^{ss_0}u_{j^*}\tilde{\lambda}_j +K(u)\\
			&\le 2\Lambda(s)\|u\|_2 + K(u)\\\
			&\le \left(2+\frac{2(1+\gamma)}{1-\gamma}\right)\Lambda(s)\|u\|_2.
		\end{aligned}
		$$
		It shows that $u$ belongs to cone $\mathcal{C}_{WSGRE}(s, s_0, \frac{2(1+\gamma)}{1-\gamma})$.
		Then, according to Assumption \ref{wsgre}, we have
		$$
		\|u\|_2^2 \le \frac{\|Xu\|_n^2}{\theta(s,s_0,\frac{2(1+\gamma)}{1-\gamma})^2}.
		$$
		Substituting \eqref{slope2} into above inequality and combining \eqref{Lambda}, we have
		$$
		\|u\|_2 \le \frac{2(1+\gamma)\Lambda(s)}{\theta(s,s_0,\frac{2(1+\gamma)}{1-\gamma})^2}\le \frac{2(1+\gamma)\sqrt{ss_0}\lambda^{\#}}{\theta(s,s_0,\frac{2(1+\gamma)}{1-\gamma})^2}.
		$$
	\end{enumerate}
	Combining these two cases, we complete the proof of Theorem \ref{sgSlopemain}.
\end{proof}

\section{Proof of minimax lower bound}

\subsection{Proof of Lemma \ref{lem3}}
\begin{proof}
	We only change the first step in \ref{lem2}, i.e., the group index set. With a little abuse of notations, we denote $\theta_{g_j} \in \mathbb{R}^p$ by only preserving the entries of group j, and setting others to 0.  In other words, $\theta = \sum_{j=1}^{m}\theta_{g_j}$. 
	
	Rethinking Lemma \ref{lem2}, let $\mathcal{M}(s)$ represent the $s$-group sparse packing set by Lemma \ref{lem2}, and $\mathcal{G}(s)$ the group index set of  $\mathcal{M}(s)$. Therefore $\forall \theta \in \mathcal{M}(s),|\{j:\theta_{g_j}\ne 0 \}| = s$; for $\forall j\in [m]$, we have $|\theta_{g_j}|_0 = s_0$. Moreover 
	$\|\theta_i,\theta_k\|_H \ge \frac{ss_0}{4}, \forall i\ne k, \theta_i,\theta_k \in \mathcal{M}(s)$. Reviewing the construction of encoding procedure, all entries of $\theta$ are 1 or 0. 
	
	
	Denote $G(\theta) = \{j:\theta_{g_j}\ne  0\}\subseteq [m]$. We prove the following result: given any $s \in [\frac{m}{2}]$, and for $\forall G \in \mathcal{G}(s)$, there exists a vector $\theta$ with the following properties:
	\begin{itemize}
		\item $G(\theta) = G$, and $\|\theta_{g_j}\|_0= s_0,\forall  G(\theta)_j = 1$;
		\item all the absolute of non-zero entries is $\frac{1}{\sqrt{ss_0}}$;
		\item $\|X\theta\|_{n} \le \vartheta_{\max}$.
	\end{itemize} 
	
	The first property holds by Lemma \ref{lem2} and the sencond is self evident. we prove third by induction. For $ s = 1$ the property is trival and we assume it is right for $s-1$, and for $s$, we give a $G \in \mathcal{G}(s)$, and $G' \subseteq G, |G'| = s-1$. So that there exists a vector $\theta$ with non-zero entries $\frac{1}{\sqrt{ss_0}}$ whose group index $G(\theta)  = G'$. We have
	$$
	\|X\theta\|_n^2 \le \vartheta_{\max}^2\frac{s-1}{s}.
	$$
	
	Now consider two vectors $\theta_1,\theta_2$ that all non-zero entries are $\frac{1}{\sqrt{ss_0}}$, subject to
	$$
	\left\{\begin{array}{ll}
		(\theta_1)_{g_j} = (\theta_2)_{g_j} = \theta_{g_j} &  j \in  G'\\
		(\theta_1)_{g_j} = -(\theta_2)_{g_j}, & j \in G \setminus G'\\
		(\theta_1)_{g_j} = (\theta_2)_{g_j} = 0, & \text{otherwise}
	\end{array}\right.
	$$
	Therefore, we have:
	$$
	\|X\theta_1\|_n^2 + \|X\theta_2\|_n^2 = 2\vartheta_{\max}^2.
	$$
	
	which means there alway a vector, $\theta_1$ or $\theta_2$ that preserve $\|X\theta\|_n \le \vartheta_{\max}$. This procedure inducates that for any group index vector $G$, there is a signed-version $\tilde{G} \in B_0(s)\cap \{-1,1,0\}^m$ making the above properties hold. So we continue the following steps the same as Lemma \ref{lem2}. Because the sign makes the hamming distance not decrease, we finish the proof.
\end{proof}

\subsection{Proof of  Theorem \ref{thm:lower}}
\begin{proof}
Consider the $\frac{ss_0}{4}$-packing set $\{\theta^1, \ldots, \theta^{M}\}$ we derive in \ref{lem3}.
Let $\beta^{(i)} = \delta\theta^i$, where $\delta$ is a parameter that need to be determined below.
For any $\beta^{(i)} \neq \beta^{(j)}$, since $\{\theta^1, \ldots, \theta^{M}\}$ is a $\frac{ss_0}{4}$-packing set of $\widetilde\Theta^{m,d}(s, s_0)$, we have 
\begin{equation}\label{ap3}
	\|\beta^{(i)}- \beta^{(j)}\|^2 \geq \frac{1}{4}ss_0\delta^2,~ \forall i,j \in [M].
\end{equation}
 On the other hand, given design matrix $X$, we have: 
 \begin{equation}\label{ap16}
 	\|X\beta^{(i)}\|_n^2 \leq \vartheta_{\max}^2 ss_0\delta^2,~ \forall i \in [M].
 \end{equation}

 Denote $y^{(i)} = X\beta^{(i)} + \xi_i, \forall i \in [M]$. We consider the Kullback-Leibler divergence
between different distribution pairs as
$$
\begin{aligned}
	KL\left(y^{(i)}|| y^{(j)}\right) &= \frac{1}{2\sigma^2}\|X(\beta^{(i)}-\beta^{(j)})\|_2^2\\
	&\le \frac{n}{\sigma^2}\left(\|X\beta^{(i)}\|_n^2 + \|X\beta^{(j)}\|_n^2 \right)\\
	&\le \frac{2n ss_0\vartheta_{\max}^2\delta^2}{\sigma^2},
\end{aligned}
$$
where the last inequality uses \eqref{ap16}. Denote $B$ as the random vector uniformly distributed over the packing set. 
Observe that
\begin{align}\label{ap6}
	\begin{split}
		I(y;B) \leq& \frac{1}{\binom{M}{2}} \sum_{i < j} KL(y^{(i)} || y^{(j)})\\
		\leq & \frac{1}{\binom{M}{2}} \sum_{i < j} \frac{1}{2\sigma^2}\|X(\beta^{(i)}-\beta^{(j)})\|_2^2\\
		\leq&\frac{2n ss_0\vartheta_{\max}^2\delta^2}{\sigma^2},
	\end{split}
\end{align}

Combining the generalized Fano's Lemma and \eqref{ap6}, we have
\begin{align*}
	\mathbf{P}(B \neq \widetilde\beta)
	&\geq 1 - \frac{I(y;B)+\log 2}{\log M}\\
	&\geq 1 - \frac{\frac{2n\vartheta_{\max}^2}{\sigma^2}ss_0\delta^2+\log 2}{\log M},
\end{align*}
where $\widetilde\beta$ takes value in the packing set.
To guarantee $P(B \neq \widetilde \beta) \geq \frac{1}{2}$, it suffices to choose
$\delta = \frac{1}{2}\sqrt{\frac{\sigma^2}{2n\vartheta_{\max}^2ss_0}\log M}$.
Substituting it into equation \eqref{ap3} and from Lemma \ref{lem2}, we have
\begin{align*}
	&\inf_{\hat{\beta}}\sup_{\beta^*\in {\Theta}^{m,d}(s, s_0) }\mathbf{P}\left(\|\hat\beta-\beta^*\|_2^2 \geq \frac{\sigma^2}{128n\vartheta_{\max}^2}( s \log \frac{em}{s}+ss_0\log \frac{ed}{s_0})\right)\\
	\geq&\inf_{\hat{\beta}}\sup_{\beta^*\in \widetilde{\Theta}^{m,d}(s, s_0) }\mathbf{P}\left(\|\hat\beta-\beta^*\|_2^2 \geq \frac{\sigma^2}{128n\vartheta_{\max}^2}( s \log \frac{em}{s}+ss_0\log \frac{ed}{s_0})\right)\\
	\geq&\mathbf{P}(B \neq \widetilde \beta)\\
	\geq&\frac{1}{2},
\end{align*}
which completes the proof of \ref{thm:lower}.
\end{proof}
\section{proof of weak random design}

\subsection{Proof of Theorem \ref{thm:ssgre}}
\begin{proof}
	Given a vector $\beta \in \mathbb{R}^{d\times m}$, it can be arranged into a group matrix of $d\times m$. where the $j$-th group is denoted as $\beta_{g_j} \in \mathbf{R}^{d}$. This proof uses $\beta_{g_j}$ to denote the parameters of the $j$-th group, and $G$ denotes the random variables to represent group index. We divide the random variable construction into four steps as follows:
	
	{\bf step 1}
	
	Construct the group indicator variable $G$:
	
	Define
	$$
	w_j = \sqrt{s_0}\|\beta_{g_j}\|_2 + \|\beta_{g_j}\|_1,\quad\forall j=1,2,\cdots,m.
	$$
	Then $W = \sum_{j=1}^m w_j$. We define $G$ to obey
	$$
	\mathbf{P}\left(G = j\right) = \frac{w_j}{W},\quad \forall j\in [m].
	$$
	Now let $G_1,G_2\cdots G_s$ be $s$ random vectors that are identically distributed with $G$ and independent of each other.

	{\bf step 2}
	
	Now given $G_k =j$. Define the conditional random variable $Y_G$:
	$$
	\mathbf{P}\left(Y_G = \left(\frac{W}{w_j}\right)\cdot\|\beta_{g_j}\|_1e_{ij}\bigg|G = j\right) = \frac{|\beta_{ij}|}{\|\beta_{g_j}\|_1},\quad \forall i\in [d],
	$$
	where the indicator variable $e_{ij}\in \{0,1\}^p$ means that $1$ is taken only at the $i$-th variable of the $j$-th group and $0$ at other positions.

	Similarly assume that $Y_{G1},Y_{G2},\cdots,Y_{Gs_0}$ are $s_0$ random vectors that are identically distributed with $Y_G$ and independent of each other. Then let $H = \frac{1}{s_0}\sum_{l=1}^{s_0}Y_{Gl}$, it is clear that $H$ has elements only on the $i$-th group and is a $s_0$-sparse vector.
	
	For convenience given $G = j$, we also define $\tilde{Y}_G = \frac{w_j}{W}Y_G$, i.e.
	$$
	\mathbf{P}\left(\tilde{Y}_G = \|\beta_{g_j}\|_1e_{ij}\bigg|G = j\right) = \frac{|\beta_{ij}|}{\|\beta_{g_j}\|_1}\quad \forall i\in [d].
	$$
	Then we have:
	$$
	\mathbf{E}\left[\tilde{Y}_G\bigg|G = j\right] = \beta_{g_j}.
	$$
	and denote $\tilde{H} = \frac{1}{s_0}\sum_{l=1}^{s_0}\tilde{Y}_{Gl}$.
	\begin{equation}\label{H}
		\begin{aligned}
			\mathbf{E}\left[\|\tilde{H}\|^2\big|G=j\right] &= \frac{1}{s_0^2}\sum_{l\ne m}\langle \mathbf{E}[\tilde{Y}_{Gl}\big|G=j],\mathbf{E}[\tilde{Y}_{Gm}\big|G =j]\rangle + \frac{1}{s_0^2}\sum_{l=1}^{s_0} \mathbf{E}[\|\tilde{Y}_{Gl}\|\big|G=j]\\\
			& = (1-\frac{1}{s_0})\|\beta_{g_j}\|_2^2 + \frac{1}{s_0}\sum_{j=1}^d\frac{|\beta_{ij}|}{\|\beta_{g_j}\|_1}\|\beta_{g_j}\|_1^2\\
			&= (1-\frac{1}{s_0})\|\beta_{g_j}\|_2^2 + \frac{1}{s_0}\|\beta_{g_j}\|_1^2.
		\end{aligned}
	\end{equation}
	
	Then for each group of schematic variables given by {\bf step 1}, we are able to generate the corresponding mutually independent conditional distributions $\mathbf{P}(H_k|G_k)$ respectively. In this way we obtain $s$ mutually independent random variables $\{H_k\}$.
	
	At this point we define:
	$$
	Z = \frac{1}{s}\sum_{k=1}^{s}H_k.
	$$

	Then at this point $ Z$ belongs to the set of doubly sparse vectors $ \mathcal{DS}(s,s_0)$. and we have $H_{k_1}$ and $H_{k_2}$ independent of each other at $k_1\ne k_2$, and:
	$$
	\begin{aligned}
		\mathbf{E}[H] &= \sum_{j=1}^{m}\frac{w_i}{W}\mathbf{E}[H|G = j]\\
		&= \sum_{j=1}^{m}\frac{w_i}{W}\mathbf{E}[H|G = j] = \sum_{j=1}^{m}\frac{w_i}{W}\mathbf{E}[Y_G|G = j]\\
		& = \sum_{j=1}^{m}\frac{w_i}{W}\sum_{j=1}^{d}\left(\frac{W}{w_i}\right)\cdot\|\beta_{g_j}\|_1e_{ij}\cdot \frac{|\beta_{ij}|}{\|\beta_{g_j j}\|_1} = \beta.
	\end{aligned}
	$$
	
	{\bf step 3}
	
	Thus under the assumptions of the DSRE condition there are:
	$$
	\mathbf{E}\|XZ\|^2 \ge \theta^2 \mathbf{E}\|Z\|^2.
	$$
	For finding $\mathbf{E}\|XZ\|^2$ we have:
	$$
	\begin{aligned}
		\mathbf{E}\|XZ\|_2^2 &= \frac{1}{s^2}\sum_{k_1\ne k_2 = 1}^{s}\mathbf{E}\langle XH_{k_1},XH_{k_2}\rangle + \frac{1}{s^2}\sum_{k=1}^{s}\mathbf{E}\|XH_k\|^2\\
		&= (1-\frac{1}{s})\|X\beta\|_2^2 + \frac{1}{s}\mathbf{E}\|XH\|_2^2.
	\end{aligned}
	$$
	
	For the second part:
	\begin{equation}
		\begin{aligned}
			\mathbf{E}\|XH\|_2^2 & = \sum_{j=1}^{m}\left(\frac{W}{w_j}\right)^2 \mathbf{E}\left[\|X\tilde{H}\|_2^2\big|G = j\right]\cdot \frac{w_j}{W}\\
			&\le \sum_{j=1}^{m}\left(\frac{W}{w_j}\right)\mathbf{E}\left[\|\tilde{H}\|_2^2\big|G = j\right]\\
			&= \sum_{j=1}^{m}\left(\frac{W}{w_j}\right)\left[(1-\frac{1}{s_0})\|\beta_{g_j}\|_2^2 + \frac{1}{s_0}\|\beta_{g_j}\|_1^2\right]\\
			&\le \frac{1}{s_0}\sum_{j=1}^{m}\left(\frac{W}{w_j}\right)\left[\sqrt{s_0}\|\beta_{g_j}\|_2 + \|\beta_{g_j}\|_1\right]^2\\
			& = \frac{1}{s_0}\sum_{j=1}^m \left(\frac{W}{w_j}\right) w_j^2 = \frac{1}{s_0}W^2.
		\end{aligned}
	\end{equation}
	The first inequality holds because $\tilde{H}$ is a $s_0$ sparse vector over a single group, so according to the sparse group normalization condition \ref{sgnorm}, we have $\|X\tilde{H}\|_2 \le \|\tilde{H}\|_2$. So that :
	\begin{equation}\label{XZ}
		\mathbf{E}\|XZ\|_2^2 \le (1-\frac{1}{s})\|X\beta\|_2^2 + \frac{1}{ss_0}W^2.
	\end{equation}
	Similarly we derive:
	\begin{equation}\label{Z}
		\mathbf{E}\|Z\|_2^2 = (1-\frac{1}{s})\|\beta\|_2^2 + \frac{1}{ss_0}W\sum_{j=1}^{m}\frac{1}{w_j}\left[(s_0-1)\|\beta_{g_j}\|_2^2+\|\beta_{g_j} \|_1^2\right].
	\end{equation}
	Combining \eqref{XZ},\eqref{Z} yields:
	$$
	\|X\beta\|_2^2 \ge \theta^2\|\beta\|_2^2 - \frac{1}{s_0(s-1)}W^2.
	$$
	Thus when $\beta$ satisfies $W^2 \le \frac{s_0(s-1)}{2}\theta^2 \|\beta\|_2^2$, we have $\|X\beta\|_2^2 \ge \frac{\theta^2}{2}\|\beta\|_2^2$. Thus when $X$ satisfies the double sparse RE condition $DSRE(s,s_0,\theta)$, we can let $\tilde{\theta} = \frac{\theta}{\sqrt{2}}$ and $\tilde{s} = \frac{(s-1)\theta^2}{2(2+c_0)^2}$, then $X$ satisfies the strong sparse group RE condition $SSGRE(\tilde{s},s_0,c_0,\tilde{\theta})$.
\end{proof}
\subsection{Proof of Theorem \ref{thm:smball}}
Here we first derive the VC-dimension of the class $\mathcal{F}(s, s_0) = \{ f(x) = x^\top \beta:\beta \in \mathcal{DS}(s, s_0)\}$, then apply the result of corollary 2.5 of \cite{lecue2017sparse}, which derives Theorem \ref{thm:smball} directly.

For brevity, $V$ denotes the VC-dimension of $\mathcal{F}(s, s_0)$.
By the definition of VC-dimension, For any fixed $ss_0$-dimension, the $VC$ of the corresponding set of $ss_0$-dimensional linear classifiers is known to be $ss_0$. 

Then, by Sauer's lemma (e.g., \cite{wainwright2019high}, Proposition 4.18) the maximal number of different labelling of $V$ points in $R ^{d_0}$ that such set of classifiers can produce is upper bounded by $\sum_{k=0}^{ss_0}\left(\begin{array}{l}V \\ k\end{array}\right)$. So considering a union of $\binom{m}{s}\binom{sd}{ss_0}$'s $ss_0$-dimensional linear classification, we have
we have 
\begin{align*}
	2^V \leq& \binom{m}{s}\binom{sd}{ss_0}\sum_{i=0}^{ss_0}\binom{V}{i}\\
        \leq& (\frac{em}{s})^s(\frac{ed}{s_0})^{ss_0}(\frac{eV}{ss_0})^{ss_0}.
\end{align*}
Denote $Z=\frac{V}{ss_0}$. By some simple algebras, we have 
\begin{align}\label{vc}
	2^Z \leq (\frac{em}{s})^s(\frac{ed}{s_0})^{ss_0}eZ.
\end{align}
We need to find the largest integer that satifies \eqref{vc}.
To guarantee this, it is sufficient to choose $Z\leq 2(\frac{1}{s_0}\log\frac{em}{s}+\log\frac{ed}{s_0})$ as an upper bound of $Z$.
Consequently, we derive the upper bound $V \leq 2(s\log\frac{em}{s}+ss_0\log\frac{ed}{s_0})$.
\subsection{Proof of weak moment normalization condition}
\begin{lemma}[Lemma 2.8 of \cite{lecue2017sparse}]\label{lecue:norm}
	There exists an absolute constant $c_0$ for which the following holds. Let $Z$ be a mean-zero random variable, and let $Z_1, Z_2,\cdots,Z_n$ be $n$ independent copies of $Z$. Let $q_0 \ge 2$ and assume that there exist $\kappa_1 > 0$ and $\alpha > \frac{1}{2}$.
	for which $\|Z\|_{L_q} \le \kappa_1 q^{\alpha},~\forall 2\le q \le q_0$. If $n \ge q_0^{\max(2\alpha-1,1)}$, then for every $2 \le q\le q_0$, we have
	$$
	\left\|\frac{1}{\sqrt{n}}\sum_{i=1}^{n}Z_i\right\|_{L_q} \le c_1(\alpha)\kappa_1\sqrt{q}, c_1(\alpha) = \exp(2\alpha -1).
	$$
\end{lemma}
\begin{proof}[proof of Theorem \ref{weak}]
	Given any $j \in \mathcal{S}$, we define $Z = (X^*_{j})^2 - \|X^*_{j}\|_{L_2}^2$. So that $E[Z] = 0$ and we have $\|Z\|_{L_q}^q \le \|X_j\|_{2q}^{2q} \le \kappa_1^{2q} (2q)^{2\alpha q},~\forall 2\le q \le q_0$, where $q_0$ will be chosen later. By Lemma \ref{lecue:norm}, if $n\ge q_0^{\max(4\alpha-1,1)}$, we have:
	
	\begin{equation}
		\left\|\frac{1}{n} \sum_{i=1}^n(X^*_{ij})^2\right\|_{L_q} \le 1 + 4\alpha c_1(2\alpha) \kappa_{1}^{2} \sqrt{\frac{q}{n}},~ \forall 2\le q \le q_0.
	\end{equation}
	
	For $\alpha,\kappa_1$ are absolute constants, we obtain the second condition $n \ge [4\alpha c_1(2\alpha) \kappa_{1}^{2}]^2 q \asymp q$, and for any $j$, we have $\left\|\frac{1}{n} \sum_{i=1}^n(X^*_{ij})^2\right\|_{L_q} \le 2$. On the other hand, we use Markov's inequality and union bound by the cardinal $|\mathcal{S}| = m\binom{d}{s_0}$:
	\begin{align*}
		\mathbf{P}\left(\max_{j \in \mathcal{S}}\frac{1}{n}\sum_{i=1}^n(X^*_{ij})^2\ge 2e\right) \le |\mathcal{S}| \left(\frac{\left\|\frac{1}{n}\sum_{i=1}^n(X^*_{ij})^2\right\|_{L_q}}{2e}\right)^q \le \exp\left(\log|\mathcal{S}|-q\right).
	\end{align*}
	
 Therefore, we obtain our conclusion by combining the three conditions above. If there exists constant $C_1 > 1$ that $q_0 \ge q > C_1(\log m + s_0\log\frac{ed}{s_0})$ holds, and sample size $n \ge q_0^{\max(4\alpha-1,1)}$, then the sparse group normalization condition \ref{sgnorm} holds with probability larger than $1-\exp[(C_2-1)(\log m + s_0\log\frac{ed}{s_0})]$.
\end{proof}
\section{proof of sub-Gaussian random design}

\subsection{Preliminaries}

At first, we define $\epsilon$-cover as follows:
\begin{definition}[$\epsilon$-cover]
     Given a set $\mathcal{T}$ and a semimetric $\rho$, there exists a set $\{\theta^1,\theta^2,\cdots,\theta^N\}\subset \mathcal{T}$. For any $\theta \in \mathcal{T}$, there exists some $i \in [N]$ such that $\rho(\theta,\theta^i)\le \epsilon$. We call this set as the $\epsilon$-cover of $\mathcal{T}$.
     The $\epsilon$-covering number $N(\epsilon;\mathcal{T},\rho)$ is the cardinality of the smallest $\epsilon$-cover of $\mathcal{T}$.
\end{definition}
Referring to the steps of \cite{zhou2009restricted}, we need the following lemma from \cite{mendelson2008uniform}, which we give without proof.
\begin{lemma}[\cite{mendelson2008uniform}, Lemma 2.2]
Given $m \ge 1$ and $\epsilon > 0$. There exists an $\epsilon$-cover $\Pi \subseteq B_2^{m}$ of $B_2^m$ with respect to the Euclidean metric such that $B_2^{m} \subseteq (1-\epsilon)^ {-1} conv(\Pi)$ and $|\Pi| \le (1 + 2/\epsilon)^m$. 
Similarly, there exists an $\epsilon$-cover $\Pi' \subseteq S^{m-1}$ of the sphere $S^{m-1}$ such that $|\Pi'| \le (1 + 2/\epsilon)^m$.
\end{lemma}
Note that the cardinality of $\mathcal{S}$ is
$
|\mathcal{S}| = m\cdot\binom{d}{s_0}.
$
We provide the lemma corresponding to the $\epsilon$-cover of $V_{s_0}$ and $\tilde{V}_{s_0}$ as follows
\begin{lemma}\label{cover}
For any given $0< \epsilon < \frac{1}{2}$ and $0 <s_0<d$, there exists an $\epsilon$-cover set $\Pi \subseteq B_2^p$ of $\tilde{V}_{s_0}$ such that $\tilde{V}_{s_0}\subseteq 2 conv(\Pi)$ and $|\Lambda|$ is at most
\begin{equation}\label{cover1}
	\quad |\Pi|\le m\cdot\binom{d}{s_0}\left(\frac{5}{2\epsilon}\right)^ {s_0}.
\end{equation}
Moreover, since $V_{s_0}\subseteq \tilde{V}_{s_0}$, $\Pi$ is also an $\epsilon$-cover set of $V_{s_0}$.
\end{lemma}
\begin{proof}
The proof is the same as Lemma 2.3 of \cite{mendelson2008uniform} and Lemma B.4 of \cite{zhou2009restricted}. Just need to replace the number of unit balls according to $|\mathcal{S}|$.
\end{proof}
Finally we need to incorporate the following lemma:
\begin{lemma}[\cite{ledoux1991probability}]
Let $X_1,X_2,\cdots,X_N$ be $p$-dimensional Gaussian random vectors. Then,
$$
\mathbf{E}\max_{i=1,2,\cdots,N}|X_i| \le 3\sqrt{\log N}\max_{i=1,2,\cdots,N}\sqrt{\mathbb{E}\|X_i\|_2^2}.
$$
\end{lemma}

\subsection{Proof of Theorem \ref{sgnorm gaussian}}
\begin{proof}
Let $g$ be a $p$-dimensional random vector with each entry draws from $\mathcal{N}(0, 1)$ independently.
Since $V_{s_0} \subseteq 2conv(\Pi)$ and according to the convexity, we have
$$
\sup_{v \in conv(\Pi)}\langle v, \Sigma^{\frac{1}{2}}g\rangle = \sup_{v \in \Pi}\langle v, \Sigma^{ \frac{1}{2}}g\rangle.
$$
Therefore,
\begin{equation}\label{tildel}
\begin{aligned}
\tilde{\ell}^*(V_{s_0})&\le \mathbf{E}\sup_{v \in 2conv(\Pi)}\langle v, \Sigma^{\frac{1}{2}}g\rangle \\
& = 2\mathbf{E}\sup_{v \in \Pi}\langle v, \Sigma^{\frac{1}{2}}g\rangle\\
&\le 6\sqrt{\log |\Pi|}\max_{S\in \mathcal{S}}\sqrt{\mathbf{E}|\langle v_S, \Sigma^{\frac{1}{2 }}g\rangle|^2}\\
& \le 6\left(\log m + s_0\log \frac{5ed}{s_0}\right)^{\frac{1}{2}}.
\end{aligned}
\end{equation}

We plug \eqref{tildel} into the Lemma \ref{mendelson1}, when satisfying the formula \eqref{nsgnorm}, for $\forall v \in V_{s_0}$, $S$ is recorded as the support set of $v$, there are:
$$
\frac{\|Xv\|_n}{\sqrt{n}\|v\|_2}=\frac{\|Xv\|_n}{\sqrt{n}\|\Sigma^{\frac{1 }{2}}_Sv\|_2}\cdot\frac{\|\Sigma^{\frac{1}{2}}_Sv\|_2}{\|v\|_2}\le (1+\theta )\cdot \frac{1}{(1+\theta)} = 1.
$$
\end{proof}

\subsection{Proof of Theorem \ref{randomwsgre}}
\begin{proof}
According to Lemma \ref{mendelson1}, we need to find a certain parameter space $V$, and derive the Gaussian complexity $\tilde{\ell}^*(V)$. We consider the parameter space as follows:
$$
V_2 = \mathcal{C}_{WSGRE}(s,s_0,c_0) \cap \{v,\|\Sigma^{\frac{1}{2}}v\|_2 = 1\}.
$$

So to define $\tilde{V}_2$, let's assume $\sigma = 1$:
\begin{equation}
\tilde{V}_2 \triangleq\left\{v:v\in R^p:\|v\|_*\le\frac{(2+c_0)}{\theta}\sqrt{\sum_{i =1}^{ss_0}\tilde{\lambda}_i^{2}},c_0 > 0,\|\Sigma^{\frac{1}{2}}u\|_2 \le 1\right\},
\end{equation}
where we can derive
\begin{equation}
\left(\sum_{i=1}^{ss_0}\tilde{\lambda}_i^2\right)^{\frac{1}{2}} =\sqrt{ss_0\log\frac{2ed}{ s_0}+ 2s\log \frac{4em}{s}}.
\end{equation}

Then $V_2\subseteq \tilde{V}_2$. Define $r = \frac{(2+c_0)}{\kappa}\sqrt{ss_0\log\frac{2ed}{s_0}+ 2s\log \frac{4em}{s}}$. Because $\Sigma$ satisfies the sparse group normalization condition, that is, for $\forall S\in \mathcal{S}$ there is $\|\Sigma_{S}\|\le 1$, so according to Lemma \ref{lem42} and Theorem \ref{gauss2}, since $\|v\|_{*} \equiv \sqrt{n}N(v)$, we have:
$$
Med\left(\sup_{v:\|v\|_{*} \le 1}\left|\langle\xi,\Sigma^{\frac{1}{2}} v\rangle \right|\right) \le 4.
$$

\begin{remark}
Combining the conclusions in the Lemma \ref{lem42} and Theorem \ref{gauss2} and the positive homogeneity of $N(u)$, we can obtain:
$$
Med\left(\sup_{v:N(v) \le \frac{1}{\sqrt{n}}}\left|\langle\xi,\frac{1}{\sqrt{n}} Xv\rangle \right|\right) \le 4.
$$

So the proof of median here is exactly the same, only need to replace $\frac{1}{\sqrt{n}}X$ with $\Sigma^{\frac{1}{2}}$ in the corresponding conditions and conclusions.
\end{remark}
Therefore for the parameter space $\tilde{V}_2$, we have:
$$
Med\left(\sup_{v \in \tilde{V}_2}\left|\langle\xi,\Sigma^{\frac{1}{2}} v\rangle \right|\right) \le 4r,
$$

and because there is $\forall v\in \tilde{V}_2\|,\Sigma^{\frac{1}{2}}v\|\le 1$ in $\tilde{V}_2$, so $f(\xi) = \sup_{v:\|v\|_{*} \in \tilde{V}_2}\left|\langle\xi,\Sigma^{\frac{1}{2} } v\rangle \right|$ is a 1-Lipschitz function. According to Lemma A.3 of \cite{bellec2018slope}:

\begin{equation}\label{wsgregauss}
\mathbf{E}\left(\sup_{v\in \tilde{V}_2}\left|\langle\xi,\Sigma^{\frac{1}{2}} v\rangle \right|\right ) \le 4r+\sqrt{\frac{2}{\pi}}.
\end{equation}

Plugging \eqref{wsgregauss} into the Lemma \ref{mendelson1}, we make the conclusion as follows: if there is a constant $c',\theta>0$ such that the sample size $n$ satisfies:
$$
n > c'\alpha^4 \theta^2\left(ss_0\log\frac{2ed}{s_0}+ s\log \frac{4em}{s}\right),
$$
then there exists $\bar{c}$, with probability at least $1 - \exp(-\bar{c}\theta^2n/\alpha^4)$, for all $v \in\mathcal{C}_ {WSGRE}(s,s_0,c_0)\setminus \{0\}$, according to homogeneity we have
$$
\frac{\|Xv\|_n}{\sqrt{n}\|\Sigma^{\frac{1}{2}}v\|_2} \ge (1-\theta),
$$
so that
$$
\frac{\|Xv\|_n}{\sqrt{n}\|v\|_2} \ge (1-\theta)\kappa.
$$
\end{proof}

\subsection{Proof of Corollary \ref{cor1}}
\begin{proof}
Because $\mathcal{C}_{SSGRE}(s,s_0,c_0)\subseteq \mathcal{C}_{WSGRE}(s,s_0,2+c_0)$, so we can immediately get the Gaussian complexity on $\mathcal{C }_{SSGRE}(s,s_0,c_0) \cap \{u:\|\Sigma^{\frac{1}{2}}u \|_2= 1\}$. The rest of the proof is exactly the same as the Theorem \ref{randomwsgre}.
\end{proof}
\end{appendix}

\bibliographystyle{unsrtnat}
\bibliography{reference}

\begin{thebibliography}{53}
\providecommand{\natexlab}[1]{#1}
\providecommand{\url}[1]{\texttt{#1}}
\expandafter\ifx\csname urlstyle\endcsname\relax
  \providecommand{\doi}[1]{doi: #1}\else
  \providecommand{\doi}{doi: \begingroup \urlstyle{rm}\Url}\fi

\bibitem[Tibshirani(1996)]{T1996}
Robert Tibshirani.
\newblock Regression shrinkage and selection via the lasso.
\newblock \emph{Journal of the Royal Statistical Society: Series B
  (Methodological)}, 58\penalty0 (1):\penalty0 267--288, 1996.
\newblock \doi{https://doi.org/10.1111/j.2517-6161.1996.tb02080.x}.
\newblock URL
  \url{https://rss.onlinelibrary.wiley.com/doi/abs/10.1111/j.2517-6161.1996.tb02080.x}.

\bibitem[Zhang(2010)]{zhang2010}
Cun-Hui Zhang.
\newblock Nearly unbiased variable selection under minimax concave penalty.
\newblock \emph{The Annals of Statistics}, 38\penalty0 (2):\penalty0 894--942,
  04 2010.
\newblock \doi{10.1214/09-AOS729}.
\newblock URL \url{https://doi.org/10.1214/09-AOS729}.

\bibitem[Raskutti et~al.(2011)Raskutti, Wainwright, and
  Yu]{raskutti2011minimax}
Garvesh Raskutti, Martin~J Wainwright, and Bin Yu.
\newblock Minimax rates of estimation for high-dimensional linear regression
  over $\ell_q$-balls.
\newblock \emph{IEEE transactions on information theory}, 57\penalty0
  (10):\penalty0 6976--6994, 2011.

\bibitem[Bellec et~al.(2018)Bellec, Lecu{\'e}, and Tsybakov]{bellec2018slope}
Pierre~C Bellec, Guillaume Lecu{\'e}, and Alexandre~B Tsybakov.
\newblock Slope meets lasso: improved oracle bounds and optimality.
\newblock \emph{The Annals of Statistics}, 46\penalty0 (6B):\penalty0
  3603--3642, 2018.

\bibitem[Yuan and Lin(2006)]{Y2006}
Ming Yuan and Yi~Lin.
\newblock Model selection and estimation in regression with grouped variables.
\newblock \emph{Journal of the Royal Statistical Society: Series B (Statistical
  Methodology)}, 68\penalty0 (1):\penalty0 49--67, 2006.
\newblock \doi{https://doi.org/10.1111/j.1467-9868.2005.00532.x}.
\newblock URL
  \url{https://rss.onlinelibrary.wiley.com/doi/abs/10.1111/j.1467-9868.2005.00532.x}.

\bibitem[Huang and Zhang(2010)]{huang2010}
Junzhou Huang and Tong Zhang.
\newblock {The benefit of group sparsity}.
\newblock \emph{The Annals of Statistics}, 38\penalty0 (4):\penalty0 1978 --
  2004, 2010.
\newblock \doi{10.1214/09-AOS778}.
\newblock URL \url{https://doi.org/10.1214/09-AOS778}.

\bibitem[Lounici et~al.(2011)Lounici, Pontil, van~de Geer, and
  Tsybakov]{tsy2011}
Karim Lounici, Massimiliano Pontil, Sara van~de Geer, and Alexandre~B.
  Tsybakov.
\newblock {Oracle inequalities and optimal inference under group sparsity}.
\newblock \emph{The Annals of Statistics}, 39\penalty0 (4):\penalty0 2164 --
  2204, 2011.
\newblock \doi{10.1214/11-AOS896}.
\newblock URL \url{https://doi.org/10.1214/11-AOS896}.

\bibitem[Zhang et~al.(2023{\natexlab{a}})Zhang, Zhu, Zhu, and Wang]{zhang2023}
Yanhang Zhang, Junxian Zhu, Jin Zhu, and Xueqin Wang.
\newblock A splicing approach to best subset of groups selection.
\newblock \emph{INFORMS Journal on Computing}, 35\penalty0 (1):\penalty0
  104--119, 2023{\natexlab{a}}.
\newblock \doi{10.1287/ijoc.2022.1241}.
\newblock URL \url{https://doi.org/10.1287/ijoc.2022.1241}.

\bibitem[Friedman et~al.(2010)Friedman, Hastie, and
  Tibshirani]{friedman2010note}
Jerome Friedman, Trevor Hastie, and Robert Tibshirani.
\newblock A note on the group lasso and a sparse group lasso.
\newblock \emph{arXiv preprint arXiv:1001.0736}, 2010.

\bibitem[Simon et~al.(2013)Simon, Friedman, Hastie, and
  Tibshirani]{simon2013sparse}
Noah Simon, Jerome Friedman, Trevor Hastie, and Robert Tibshirani.
\newblock A sparse-group lasso.
\newblock \emph{Journal of computational and graphical statistics}, 22\penalty0
  (2):\penalty0 231--245, 2013.

\bibitem[Ida et~al.(2019)Ida, Fujiwara, and Kashima]{ida2019fast}
Yasutoshi Ida, Yasuhiro Fujiwara, and Hisashi Kashima.
\newblock Fast sparse group lasso.
\newblock \emph{Advances in Neural Information Processing Systems}, 32, 2019.

\bibitem[Zhang et~al.(2020)Zhang, Zhang, Sun, and Toh]{zhang2020efficient}
Yangjing Zhang, Ning Zhang, Defeng Sun, and Kim-Chuan Toh.
\newblock An efficient hessian based algorithm for solving large-scale sparse
  group lasso problems.
\newblock \emph{Mathematical Programming}, 179\penalty0 (1):\penalty0 223--263,
  2020.

\bibitem[Chatterjee et~al.(2012)Chatterjee, Steinhaeuser, Banerjee, Chatterjee,
  and Ganguly]{chatterjee2012sparse}
Soumyadeep Chatterjee, Karsten Steinhaeuser, Arindam Banerjee, Snigdhansu
  Chatterjee, and Auroop Ganguly.
\newblock Sparse group lasso: Consistency and climate applications.
\newblock In \emph{Proceedings of the 2012 SIAM International Conference on
  Data Mining}, pages 47--58. SIAM, 2012.

\bibitem[Rao et~al.(2013)Rao, Cox, Nowak, and Rogers]{rao2013sparse}
Nikhil Rao, Christopher Cox, Rob Nowak, and Timothy~T Rogers.
\newblock Sparse overlapping sets lasso for multitask learning and its
  application to fmri analysis.
\newblock \emph{Advances in neural information processing systems}, 26, 2013.

\bibitem[Ahsen and Vidyasagar(2017)]{ahsen2017error}
M~Eren Ahsen and Mathukumalli Vidyasagar.
\newblock Error bounds for compressed sensing algorithms with group sparsity: A
  unified approach.
\newblock \emph{Applied and Computational Harmonic Analysis}, 43\penalty0
  (2):\penalty0 212--232, 2017.

\bibitem[Poignard(2020)]{poignard2020asymptotic}
Benjamin Poignard.
\newblock Asymptotic theory of the adaptive sparse group lasso.
\newblock \emph{Annals of the Institute of Statistical Mathematics},
  72\penalty0 (1):\penalty0 297--328, 2020.

\bibitem[Tony~Cai et~al.(2022)Tony~Cai, Zhang, and Zhou]{cai2019sparse}
T.~Tony~Cai, Anru~R. Zhang, and Yuchen Zhou.
\newblock Sparse group lasso: Optimal sample complexity, convergence rate, and
  statistical inference.
\newblock \emph{IEEE Transactions on Information Theory}, 68\penalty0
  (9):\penalty0 5975--6002, 2022.
\newblock \doi{10.1109/TIT.2022.3175455}.

\bibitem[Li et~al.(2022)Li, Zhang, and Yin]{li2022minimax}
Zhifan Li, Yanhang Zhang, and Jianxin Yin.
\newblock Minimax rates for high-dimensional double sparse structure over
  $\ell_u(\ell_q)$-balls.
\newblock \emph{arXiv preprint arXiv:2207.11888}, 2022.

\bibitem[Wainwright(2009)]{wainwright2009sharp}
Martin~J Wainwright.
\newblock Sharp thresholds for high-dimensional and noisy sparsity recovery
  using $\ell_1 $-constrained quadratic programming (lasso).
\newblock \emph{IEEE transactions on information theory}, 55\penalty0
  (5):\penalty0 2183--2202, 2009.

\bibitem[Meinshausen and Yu(2009)]{meinshausen2009lasso}
Nicolai Meinshausen and Bin Yu.
\newblock Lasso-type recovery of sparse representations for high-dimensional
  data.
\newblock \emph{The annals of statistics}, 37\penalty0 (1):\penalty0 246--270,
  2009.

\bibitem[Wainwright et~al.(2006)Wainwright, Lafferty, and
  Ravikumar]{wainwright2006high}
Martin~J Wainwright, John Lafferty, and Pradeep Ravikumar.
\newblock High-dimensional graphical model selection using $\ell_1$-regularized
  logistic regression.
\newblock \emph{Advances in neural information processing systems}, 19, 2006.

\bibitem[Raskutti et~al.(2008)Raskutti, Yu, Wainwright, and
  Ravikumar]{raskutti2008model}
Garvesh Raskutti, Bin Yu, Martin~J Wainwright, and Pradeep Ravikumar.
\newblock Model selection in gaussian graphical models: High-dimensional
  consistency of $\ell_1$-regularized mle.
\newblock \emph{Advances in Neural Information Processing Systems}, 21, 2008.

\bibitem[Ravikumar et~al.(2010)Ravikumar, Wainwright, and
  Lafferty]{ravikumar2010high}
Pradeep Ravikumar, Martin~J Wainwright, and John~D Lafferty.
\newblock High-dimensional ising model selection using $\ell_1$-regularized
  logistic regression.
\newblock \emph{The Annals of Statistics}, 38\penalty0 (3):\penalty0
  1287--1319, 2010.

\bibitem[Bickel et~al.(2009)Bickel, Ritov, and
  Tsybakov]{bickel2009simultaneous}
Peter~J Bickel, Ya’acov Ritov, and Alexandre~B Tsybakov.
\newblock Simultaneous analysis of lasso and dantzig selector.
\newblock \emph{The Annals of statistics}, 37\penalty0 (4):\penalty0
  1705--1732, 2009.

\bibitem[Cand{\`e}s et~al.(2006)Cand{\`e}s, Romberg, and Tao]{candes2006robust}
Emmanuel~J Cand{\`e}s, Justin Romberg, and Terence Tao.
\newblock Robust uncertainty principles: Exact signal reconstruction from
  highly incomplete frequency information.
\newblock \emph{IEEE Transactions on information theory}, 52\penalty0
  (2):\penalty0 489--509, 2006.

\bibitem[Cand{\`e}s and Tao(2007)]{candes2007dantzig}
Emmanuel Cand{\`e}s and Terence Tao.
\newblock The dantzig selector: Statistical estimation when p is much larger
  than n.
\newblock \emph{The annals of Statistics}, 35\penalty0 (6):\penalty0
  2313--2351, 2007.

\bibitem[B{\"u}hlmann and Van De~Geer(2011)]{buhlmann2011statistics}
Peter B{\"u}hlmann and Sara Van De~Geer.
\newblock \emph{Statistics for high-dimensional data: methods, theory and
  applications}.
\newblock Springer Science \& Business Media, 2011.

\bibitem[Bogdan et~al.(2015)Bogdan, Van Den~Berg, Sabatti, Su, and
  Cand{\`e}s]{bogdan2015slope}
Małgorzata Bogdan, Ewout Van Den~Berg, Chiara Sabatti, Weijie Su, and
  Emmanuel~J Cand{\`e}s.
\newblock Slope—adaptive variable selection via convex optimization.
\newblock \emph{The annals of applied statistics}, 9\penalty0 (3):\penalty0
  1103, 2015.

\bibitem[Su and Cand{\`e}s(2016)]{su2016slope}
Weijie Su and Emmanuel Cand{\`e}s.
\newblock Slope is adaptive to unknown sparsity and asymptotically minimax.
\newblock \emph{The Annals of Statistics}, 44\penalty0 (3):\penalty0
  1038--1068, 2016.

\bibitem[Abramovich et~al.(2006)Abramovich, Benjamini, Donoho, and
  Johnstone]{abramovich2006adapting}
Felix Abramovich, Yoav Benjamini, David~L Donoho, and Iain~M Johnstone.
\newblock Adapting to unknown sparsity by controlling the false discovery rate.
\newblock \emph{The Annals of Statistics}, 34\penalty0 (2):\penalty0 584--653,
  2006.

\bibitem[Abramovich et~al.(2007)Abramovich, Grinshtein, and
  Pensky]{abramovich2007optimality}
Felix Abramovich, Vadim Grinshtein, and Marianna Pensky.
\newblock On optimality of bayesian testimation in the normal means problem.
\newblock \emph{The Annals of Statistics}, 35\penalty0 (5):\penalty0
  2261--2286, 2007.

\bibitem[Wu and Zhou(2013)]{wu2013model}
Zheyang Wu and Harrison~H Zhou.
\newblock Model selection and sharp asymptotic minimaxity.
\newblock \emph{Probability Theory and Related Fields}, 156\penalty0
  (1):\penalty0 165--191, 2013.

\bibitem[Brzyski et~al.(2019)Brzyski, Gossmann, Su, and Bogdan]{groupslope}
Damian Brzyski, Alexej Gossmann, Weijie Su, and Małgorzata Bogdan.
\newblock Group slope – adaptive selection of groups of predictors.
\newblock \emph{Journal of the American Statistical Association}, 114\penalty0
  (525):\penalty0 419--433, 2019.
\newblock \doi{10.1080/01621459.2017.1411269}.
\newblock URL \url{https://doi.org/10.1080/01621459.2017.1411269}.
\newblock PMID: 31217649.

\bibitem[Lecu{\'e} and Mendelson(2017)]{lecue2017sparse}
Guillaume Lecu{\'e} and Shahar Mendelson.
\newblock Sparse recovery under weak moment assumptions.
\newblock \emph{Journal of the European Mathematical Society}, 19\penalty0
  (3):\penalty0 881--904, 2017.

\bibitem[Zhang et~al.(2023{\natexlab{b}})Zhang, Li, and Yin]{zhang2023minimax}
Yanhang Zhang, Zhifan Li, and Jianxin Yin.
\newblock A minimax optimal approach to high-dimensional double sparse linear
  regression.
\newblock \emph{arXiv preprint arXiv:2305.04182}, 2023{\natexlab{b}}.

\bibitem[Verzelen(2012)]{verzelen2012minimax}
Nicolas Verzelen.
\newblock Minimax risks for sparse regressions: Ultra-high dimensional
  phenomenons.
\newblock \emph{Electronic Journal of Statistics}, 6:\penalty0 38--90, 2012.

\bibitem[Gilbert(1952)]{gilbert1952comparison}
Edgar~N Gilbert.
\newblock A comparison of signalling alphabets.
\newblock \emph{The Bell system technical journal}, 31\penalty0 (3):\penalty0
  504--522, 1952.

\bibitem[Zhou(2009)]{zhou2009restricted}
Shuheng Zhou.
\newblock Restricted eigenvalue conditions on subgaussian random matrices.
\newblock \emph{arXiv preprint arXiv:0912.4045}, 2009.

\bibitem[Oliveira(2016)]{oliveira2016lower}
Roberto~Imbuzeiro Oliveira.
\newblock The lower tail of random quadratic forms with applications to
  ordinary least squares.
\newblock \emph{Probability Theory and Related Fields}, 166:\penalty0
  1175--1194, 2016.

\bibitem[Koltchinskii and Mendelson(2015)]{koltchinskii2015bounding}
Vladimir Koltchinskii and Shahar Mendelson.
\newblock Bounding the smallest singular value of a random matrix without
  concentration.
\newblock \emph{International Mathematics Research Notices}, 2015\penalty0
  (23):\penalty0 12991--13008, 2015.

\bibitem[Mendelson(2015)]{mendelson2015learning}
Shahar Mendelson.
\newblock Learning without concentration.
\newblock \emph{Journal of the ACM (JACM)}, 62\penalty0 (3):\penalty0 1--25,
  2015.

\bibitem[Mendelson et~al.(2008)Mendelson, Pajor, and
  Tomczak-Jaegermann]{mendelson2008uniform}
Shahar Mendelson, Alain Pajor, and Nicole Tomczak-Jaegermann.
\newblock Uniform uncertainty principle for bernoulli and subgaussian
  ensembles.
\newblock \emph{Constructive Approximation}, 28:\penalty0 277--289, 2008.

\bibitem[Ledoux and Talagrand(1991)]{ledoux1991probability}
Michel Ledoux and Michel Talagrand.
\newblock \emph{Probability in Banach Spaces: isoperimetry and processes},
  volume~23.
\newblock Springer Science \& Business Media, 1991.

\bibitem[Lepskii(1991)]{lepskii1991problem}
OV~Lepskii.
\newblock On a problem of adaptive estimation in gaussian white noise.
\newblock \emph{Theory of Probability \& Its Applications}, 35\penalty0
  (3):\penalty0 454--466, 1991.

\bibitem[Lepski and Spokoiny(1997)]{lepski1997optimal}
Oleg~V Lepski and Vladimir~G Spokoiny.
\newblock Optimal pointwise adaptive methods in nonparametric estimation.
\newblock \emph{The Annals of Statistics}, 25\penalty0 (6):\penalty0
  2512--2546, 1997.

\bibitem[Dalalyan and Minasyan(2022)]{dalalyan2022all}
Arnak~S Dalalyan and Arshak Minasyan.
\newblock All-in-one robust estimator of the gaussian mean.
\newblock \emph{The Annals of Statistics}, 50\penalty0 (2):\penalty0
  1193--1219, 2022.

\bibitem[Aeckerle-Willems and Strauch(2022)]{aeckerle2022sup}
Cathrine Aeckerle-Willems and Claudia Strauch.
\newblock Sup-norm adaptive drift estimation for multivariate nonreversible
  diffusions.
\newblock \emph{The Annals of Statistics}, 50\penalty0 (6):\penalty0
  3484--3509, 2022.

\bibitem[Oymak et~al.(2015)Oymak, Jalali, Fazel, Eldar, and
  Hassibi]{oymak2015simultaneously}
Samet Oymak, Amin Jalali, Maryam Fazel, Yonina~C Eldar, and Babak Hassibi.
\newblock Simultaneously structured models with application to sparse and
  low-rank matrices.
\newblock \emph{IEEE Transactions on Information Theory}, 61\penalty0
  (5):\penalty0 2886--2908, 2015.

\bibitem[Hao et~al.(2020)Hao, Zhang, and Cheng]{hao2020sparse}
Botao Hao, Anru~R Zhang, and Guang Cheng.
\newblock Sparse and low-rank tensor estimation via cubic sketchings.
\newblock In \emph{International Conference on Artificial Intelligence and
  Statistics}, pages 1319--1330. PMLR, 2020.

\bibitem[Zhang and Han(2019)]{zhang2019optimal}
Anru Zhang and Rungang Han.
\newblock Optimal sparse singular value decomposition for high-dimensional
  high-order data.
\newblock \emph{Journal of the American Statistical Association}, 114\penalty0
  (528):\penalty0 1708--1725, 2019.

\bibitem[van~de Geer(2000)]{geer2000empirical}
Sara van~de Geer.
\newblock \emph{Empirical Processes in M-estimation}, volume~6.
\newblock Cambridge university press, 2000.

\bibitem[Boucheron et~al.(2013)Boucheron, Lugosi, and
  Massart]{boucheron2013concentration}
St{\'e}phane Boucheron, G{\'a}bor Lugosi, and Pascal Massart.
\newblock \emph{Concentration inequalities: A nonasymptotic theory of
  independence}.
\newblock Oxford university press, 2013.

\bibitem[Wainwright(2019)]{wainwright2019high}
Martin~J Wainwright.
\newblock \emph{High-dimensional statistics: A non-asymptotic viewpoint},
  volume~48.
\newblock Cambridge University Press, 2019.

\end{thebibliography}

\end{document}